\documentclass{article}

\usepackage{arxiv}

\usepackage[utf8]{inputenc} 
\usepackage[T1]{fontenc}    
\usepackage{hyperref}       
\usepackage{url}            
\usepackage{booktabs}       
\usepackage{amsfonts}       
\usepackage{nicefrac}       
\usepackage{microtype}      
\usepackage{lipsum}
\usepackage{graphicx}
\usepackage{color}
\usepackage{caption}
\usepackage{amsthm}
\usepackage{amssymb}
\usepackage{amsmath}

\newtheorem{definition}{Definition}[section]
\newtheorem{theorem}{Theorem}[section]
\newtheorem{proposition}{Proposition}[section]

\newtheorem{example}{Example}[section]
\newtheorem{lemma}{Lemma}[section]

\newtheorem{corollary}{Corollary}[section]
\newtheorem{remark}{Remark}[section]

\title{A Novel Approach for Computing Hilbert Functions }

\author{
	Maria Barouti \\
	American University\\
	Dept. of Mathematics and Statistics\\
	Washington, DC 20016\\
	\texttt{barouti@american.edu} 
}

\begin{document}
\maketitle

\begin{abstract}
The Hilbert function for any graded module, $M$, over a field $k$, is defined by 
${\rm HF}(M,b)={\rm dim}_k M_b$,
where integer $b$ indicates the graded component being considered. 
One standard approach to compute the Hilbert function is to come up with a free-resolution for the graded module $M$ and another is via a Hilbert power series which serves as a generating function.  The proposed approaches enable generating the values of a Hilbert function when the graded module is a quotient ring over a field by using combinatorics and homological algebra.
Two of these approaches named the lcm-Lattice method and the Syzygy method, are conceptually combinatorial and work for any polynomial quotient ring over a field. The third approach named Hilbert function table method, also uses syzygies but the approach is better described in terms of homological algebra.
\end{abstract}

\section{Introduction}
In this work we address only polynomial rings and their quotient rings.  Therefore, all definitions pertaining a ring are meant to apply to \emph{commutative} rings.  As a consequence, all our modules are two-sided modules and all our ideals are two-sided ideals.  In fact, we require more structure of our objects: we require that they should be graded objects.  This is made precise by the first two definitions.

\begin{definition}
A graded ring $R$ is a ring that has a direct sum decomposition into abelian additive groups
$R= \displaystyle \bigoplus _{n\in \mathbb{Z}^{\geq 0}} R_n=R_0 \bigoplus {R_1} \bigoplus {R_2} \bigoplus {R_3} \bigoplus {...}$, such that $R_{s} R_{r} \subseteq R_{s+r}$  for all $r, s \geq 0$.
\end{definition}
There is also the closely related concept of a graded module.
\begin{definition}
A graded module $M$ over any graded ring $R$ is a module that can be written as a direct sum $M=  \displaystyle \bigoplus _{i\in \mathbb{Z}^{\geq 0}} M_{i}$ satisfying $R_{i} M_{j} \subseteq M_{i+j}$ for all $i, j \geq 0$.
\end{definition}
Both concepts of a graded object are standard; see for example \cite{villarreal2015monomial} pages 12 and 13. An example of a graded ring and also of a graded module is the polynomial ring $k[x_1,x_2,...,x_a]$ over a field $k$. The direct decomposition in this case is  $R=k[x_1,x_2,...,x_a]= \displaystyle \bigoplus _{b\in \mathbb{Z}^{\geq 0}}R_b$,  where each $R_b={\rm span}_{k} \left\{{\rm monomials\ of\ degree}\ b\right\}$. This means that each $R_b$ is a $k$-vector space.  Moreover, since every ideal $I$ of a ring $R$ is an $R$-module one can easily prove the following result.

\begin{lemma}\label{L:characterization_graded_ideal}
An ideal $I$ is a graded ideal of a graded ring $R=\displaystyle \bigoplus _{n\in \mathbb{Z}^{\geq 0}} R_n$ if it can be written as a direct sum of ideals such that each summand corresponds to $I \cap R_n$ for $n \in \mathbb{Z}^{\geq 0}$.
\end{lemma}

\begin{definition}
An ideal I of $k[x_1,x_2,...,x_a]$ is homogeneous if and only if every homogeneous component of every polynomial $p(\bar {{\rm x}})$ is in $I$, where $\bar {{\rm x}}$ denotes an a-tuple $(x_1,x_2,....x_a)$  {\rm (see for example \cite{dummit2004abstract} page 299)}.
\end{definition}

Here are some easy-to-prove facts relevant to the present discussion about monomial ideals.
\begin{itemize}
\item A monomial ideal in $k[x_1,x_2,...,x_a]$ is, by definition, one generated by monomials (see \cite{dummit2004abstract} page 318).  Therefore, it is a homogeneous ideal since every monomial is a homogeneous polynomial.  A monomial ideal, is also, a graded ideal because  $k[x_1,x_2,...,x_a]$ is a graded ring. Hence we may apply Lemma \ref{L:characterization_graded_ideal}.
\item $R/I$ is a graded module since it has the following direct sum decomposition 
\begin{equation}
R/I=  \displaystyle \bigoplus _{b\in \mathbb{Z}^{\geq 0}} (R_{b}+I)/I,
\end{equation}
 where $b$ is the grading and $I$ is a monomial ideal in the polynomial ring $R$. Observe that every summand is also a module over the base field $k$ of polynomial ring $R$.
\end{itemize}

Our object of study is the graded modules $R/I$, where $R$ is a polynomial ring in finitely many variables over a field $k$ and $I$ is a finitely generated monomial ideal in $R$. In this setting, we have that for each $b \geq 0$ the summand $(R_{b}+I)/I$ is indeed a vector space since it is a module over a field. Furthermore, since the number of variables is finite each such summand is a finite dimensional vector space. This brings up a natural question: Given a summand with grading $b$, what is its dimension as a vector space over the base field $k$? This is in fact how the Hilbert function for the graded module is defined.

This definition can be illustrated by considering $R=k[x_1,x_2,x_3,x_4]$ and $I=\langle x_2^{4},x_{1}x_{4},x_3^2\rangle$. Then $R/I=\bigoplus_{i=0}^{\infty}R_i$, where $R_i=\left\{{\rm all\ polynomials\ equivalence\ classes\ in\ }R/I\ {\rm with\ representatives\ of\ degree\ i }\right\}$. Each $R_i$ is no longer a ring on its own but it is a $k$-vector space.  The dimension of these vectors spaces are, $\dim R_0 = 1$, $\dim R_1 = 4$, $\dim R_2 =8$, $\dim R_3 =12$, $\dim R_4=15$, and $\dim R_i=16$ for all $i \geq 5$. In general, we define the Hilbert function of $M$ as ${\rm HF}(M,b)={\rm dim}_k M_b$  for any graded module $M=  \displaystyle \bigoplus _{i\in \mathbb{N}} M_{i}$.
In particular, a basic result facilitating our computations is the ``rank-nullity'' theorem.
\begin{theorem}\label{rank-nullity Hilbert}
${\rm HF}(R,b)={\rm HF}(R/I,b)+{\rm HF}(I,b)$, where  $R=k[x_1,x_2,...,x_a]$  and $I$ is a monomial ideal.
\end{theorem}
\begin{proof}
Let $T: V \longrightarrow W$ be a linear transformation; then using the inclusion map $i$ of the kernel \emph{into} $V$ we get a sequence of linear transformations:
\begin{center}
$0\longrightarrow {\rm ker(T)}\overset{i}\longrightarrow V \overset{T}\longrightarrow {\rm coker(T)} \longrightarrow 0$.
\end{center}
If $T$ is onto then ${\rm coker(T)} \cong {V/ {\rm ker(T)}}$. Then
\begin{center}
 ${\rm dim(ker(T))} - {\rm dim(V)} + {\rm dim(coker(T))}=0$.
\end{center}
\end{proof}
\begin{definition}
An exact sequence of modules is either a finite or an infinite sequence of modules and homomorphisms between them such that the image of one homomorphism equals the kernel of the next homomorphism {\rm (see \cite{dummit2004abstract} page 378)}.
\end{definition}
An example of an exact sequence is the sequence in the next lemma (see \cite{villarreal2015monomial} page 98) and the free resolution used in the Hilbert Syzygy theorem below (see \cite{eisenbud2013commutative} page 3). We shall refer to the exact sequence in the next lemma as \emph{the short exact sequence}.
Bookkeeping often requires a shift in the grading.  If $M=  \displaystyle \bigoplus _{i=0}^{\infty} M_i$ is a finitely generated $\mathbb{Z}^{\geq 0}$-graded module over $R$, then we denote $M(-d)$ to be the regrading of $M$ obtained by a shift of the grading of $M$. In this case, the graded component $M_i$ of $M$ becomes $M_{i+d}$ grading component of $M(-d)$.

 \begin{lemma}
Let M be a graded $R$-module. If $x_n \in R_d$ with ${\rm deg(x_n)}=d$, then there is a degree preserving exact sequence 
\begin{center}
$0 \rightarrow (0:x_n)(-d) \rightarrow M(-d) \overset{x_n}{\rightarrow} M \overset{\phi}{\rightarrow} M/x_{n}M \rightarrow 0$,
\end{center}
where $\phi(m)=m+x_{n}M$ and $(0:x_n)=\{m \in M |x_{n}m=0\}$.
\end{lemma}
The drawback of this sequence is that not all objects are necessarily free $R$-modules. Free $R$-modules are isomorphic to a direct sum of copies of $R$. The traditional approach (see \cite{W:B&W}) to compute the Hilbert function of a finitely graded $R$-module $M$ (of which our quotient polynomial rings are examples) is based on the following theorem (see page 45 of \cite{eisenbud2013commutative}).

\begin{theorem}{\rm (Hilbert Syzygy Theorem)}\\
Any finitely generated module $M$ over the ring $R=k[x_1,x_2,...,x_a]$ has a finite graded free resolution 
\begin{center}
$0\rightarrow P_n  \overset{\phi_n} {\rightarrow} P_{n-1} \rightarrow ... {\rightarrow} P_1  \overset{\phi_1}\rightarrow P_0$.
\end{center}
This implies that each $P_i$ is a finitely generated free $R$-module and $M \cong P_0/\ker \phi_1$.  Furthermore, $n \leq a$.
\end{theorem}
This exact sequence can also be written as
\begin{center}
$0\rightarrow P_n  \overset{\phi_n} {\rightarrow} P_{n-1} \rightarrow ... {\rightarrow} P_1  \overset{\phi_1}\rightarrow P_0 \rightarrow M \rightarrow 0$
\end{center}
since each $P_i$ is a free $R$-module for $0 \leq i \leq n$. If $R$ is a graded ring, the sequence above is in fact an exact sequence of graded free modules and graded homomorphism, where each term in the free resolution is of the form $P_i=R_{1_i}(- d_{1_i}) \oplus R_{2_i}(- d_{2_i}) \oplus \dots \oplus R_{l_i}(- d_{l_i})$. Then by applying Theorem \ref{rank-nullity Hilbert} in an inductive argument one obtains the following method for computing $HF(M,t)$
\[
{\rm HF}(M,t)=\sum_{i=0}^n (-1)^i \left( {\rm HF}(R_{1_i}(- d_{1_i}),t) + {\rm HF}(R_{2_i}(- d_{2_i}),t) + \dots + {\rm HF}(R_{l_i}(- d_{l_i}),t)\right).
\]
Another standard approach to compute the Hilbert function is via the Hilbert series.
\begin{definition}
Let $R= \displaystyle \bigoplus R_n$ be a graded ring. The Hilbert series of $R$ is defined to be the generating function 
\begin{center}
${\rm HS}(R,t)=\displaystyle\sum_{n=0}^{\infty} {\rm HF}(R,n)t^{n}$.
\end{center}
\end{definition}
Similarly, if $I$ is a homogeneous ideal of $R$, then the Hilbert series of $I$ is the formal power series  
\begin{center}
 ${\rm HS}(I,t)=\displaystyle\sum_{n=0}^{\infty} {\rm HF}(I,n)t^{n}$. 
\end{center}
Convergence is not an issue since we are working with formal power series. For the Hilbert series we have a counterpart to our result derived from the ``rank-nullity'' theorem.
\begin{theorem}
Let $R= \displaystyle \bigoplus _{n\geq 0} R_n$ be a graded ring and $I= \displaystyle \bigoplus _{n\geq 0} I_n$ be a graded ideal. Then 
\begin{center}
${\rm HS}(R/I,t)={\rm HS}(R,t)-{\rm HS}(I,t)$.
\end{center}
\end{theorem}

\begin{proof}
Theorem \ref{rank-nullity Hilbert} implies that ${\rm HF}(R/I,n)={\rm HF}(R,n)-{\rm HF}(I,n)$ and by summing over all values of $n$ the theorem follows.
\end{proof}
In other words, for computing the dimension of $R_n/I_n$, we count the number of monomials in $R_n$ and we subtract the number of monomials spanning $I_n$; this is because the monomials spanning $R_n$ form a basis for  $R_n$ as a vector space over $k$. Similarly the monomials spanning  $I_n$ form a basis for  $I_n$ as a vector space over $k$. \\ To build on this result we need the following notation for the Hilbert function of a module $M$ shifted by degree $d$
\begin{center}
${\rm HF}\{M(-d)\}:={\rm HF}(M,t-d)$.
\end{center}

\begin{lemma}\label{HF for principal ideal}
A principal ideal has the Hilbert function of a polynomial ring shifted by the degree of the generator. If $I=\langle p \rangle$, where $p$  is a monomial of degree $n$ in $k[\bar {\rm x}]$ and $\bar {\rm x}$ represents the a-tuple $(x_1,x_2,x_3,...,x_a)$ then 
\begin{center}
${\rm HF}(I,t)={\rm HF}\{k[\bar  {\rm x}](-n)\}$.
\end{center}
\end{lemma}

\begin{proof}
By definition ${\rm HF(I,t)}$ is the dimension of the vector space spanned by all polynomials in $I$ of uniform degree $t$. A basis for such a vector space can be chosen to be all monomials in $I$ of degree $t$. These are of the form $f \cdot p $, where $f$ is a monomial of ${\rm deg}(f)=t-{\rm deg}(p)$ so there are as many such monomials as there are of degree $t-n$ in $k[\bar  {\rm x}](-n)$.
\end{proof}

Before working through our first example, it would be helpful to refer the following corollary to our last lemma.

\begin{corollary}\label{corollary to HF for principal ideal}
For a principal ideal $I=\langle p \rangle$ we have that 
\begin{center}
${\rm HF}(R/I,t)={\rm HF}(R,t)-{\rm HF}(R(- {\rm deg}(p)),t)$ 
\end{center}
\end{corollary}
\begin{proof}
Apply the above lemma to the Theorem \ref{rank-nullity Hilbert}
\end{proof}

\begin{example} 
Find the Hilbert function of  $M=k[x,y,z]/\langle x^5 \rangle$.
\end{example}

Let $R=k[x,y,z]$. By Corollary \ref{corollary to HF for principal ideal}, the Hilbert function of the module $M$ can be written as 
\begin{center}
${\rm HF}(M,t)={\rm HF}(R,t)-{\rm HF}(R(- {\rm deg}(x^5)),t)={\rm HF}(R,t)-{\rm HF}(R(- 5),t)$.
\end{center} 
Therefore,
\begin{center}
\begin{tabular}{ l | l | l }
${\rm HF}\{R\}$ & $-{\rm HF}\{R\} (-5)\}$ & ${\rm HF}\{M\}$\\
\hline
 ~1 & ~~~~~0 &~~~ $\bf 1$ \\
\hline
 ~3 & ~~~~~0 &~~~ $\bf 3$\\
\hline
 ~6 & ~~~~~0 &~~~ $\bf 6$ \\
\hline
 10 & ~~~~~0 &~~ $\bf 10$ \\
\hline
 15 & ~~~~~0 &~~ $\bf 15$ \\
\hline
 21 & ~~~~-1 &~~ $\bf 20$ \\
\hline
 28 & ~~~~-3 &~~ $\bf 25$ \\
\hline
 36 & ~~~~-6 &~~ $\bf 30$ \\
\hline
 45 & ~~~-10 &~~ $\bf 35$ \\
\hline
 55 & ~~~-15 &~~ $\bf 40$ \\
\hline
 ~.. & ~~~~.. &~~~~~.. \\
\hline
 ~.. & ~~~~..  &~~~~~.. \\
\end{tabular}
\end{center}

Regardless of our approach to the Hilbert function of polynomial quotient rings, it is clear that computing the Hilbert function of rings of the form $k[x_1,x_2,...,x_a]$ is essential.

\section{Hilbert Function tables, motivating applications and examples}

We study Hilbert functions by placing them into families.  The simplest such family will be the Hilbert functions corresponding to the indexed set $\{k[x_1, x_2, \ldots, x_a]\, : \, a \geq 1\}$.  Then we generalize the idea of the Pascal table to construct the Hilbert Function tables.  To motivate this generalization, we use the Stanley--Reisner ring of a complex which we gradually build in a form that is analogous to the way the corresponding Hilbert Function table would be generated.  Finally, one must address the difficulties of generating a row of the Hilbert Function table which involves the introduction of one or more monomials in the ideal being used for the quotient ring corresponding to that row.  We illustrate the difficulties at the end of this section and develop a different method of solving this problem in each of the next two sections.

\subsection{Pascal Table and more general Hilbert Function Tables}
Consider  the indexed set $\{k[x_1, x_2, \ldots, x_a]\, : \, a \geq 1\}$ of polynomial rings.  We use the index value $a$ to determine the row and the degree $b$ of the monomials being counted to determine the column in the table below.
\begin{center}
\small{
\begin{tabular}{ l | l l l l l l l l l}
 HF of $k[x_1]$ & 1 & 1 & 1 & 1 & 1 & 1 & 1 & 1 &...\\ [-3pt]
\hline
 HF of $k[x_1,x_2]$ & 1 & 2& 3 & 4 & 5 & 6 & 7 & 8 &... \\[-3pt]
\hline
 HF of $k[x_1,x_2,x_3]$ & 1 & 3 & 6 & 10 & 15 & 21 & 28 & 36 &...\\[-3pt]
\hline
 HF of $k[x_1,x_2,x_3,x_4]$ & 1 & 4 & 10 & 20 & 35 & 56 & 84 & 120&... \\[-3pt]
\hline
 HF of $k[x_1,x_2,x_3,x_4,x_5]$ & 1 & 5 & 15 & 35 & 70 & 126 & 210 & 330&... \\[-3pt]
\hline
 HF of $k[x_1,x_2,x_3,x_4,x_5,x_6]$ & 1 & 6 & 21 & 56 & 126 & 252 & 462 & 792&... \\[-3pt]
\hline
 HF of $k[x_1,x_2,x_3,x_4,x_5,x_6,x_7]$ & 1 & 7 & 28 & 84 & 210 & 462 & 924 & 1716&... \\[-3pt]
\hline
 HF of $k[x_1,x_2,x_3,x_4,x_5,x_6,x_7,x_8]$ & 1 & 8 & 36 & 120 & 330 & 792 & 1716 & 3432&... \\[-3pt]
\hline
 . & . & . & . & . & . & . & . & .&... \\[-3pt]
\end{tabular}}
\end{center}

The reader would have undoubtedly noticed that the number patterns displayed in the above table are  those of the Pascal triangle.  For this reason, we refer to the above table as the Pascal table.  These numerical patterns lead us to the following proposition. 
\begin{proposition}\label{Pascal_table_recurrence_formula}
$F(a,b)= F(a-1,b)+F(a,b-1)$, where $F(a,b)$ denotes the number of monomials of degree $b$ in $k[x_1,x_2,...,x_a]$.
\end{proposition}
\begin{proof}
Let $S$ be the set of monomials in $k[x_1,x_2,....,x_a]$ of degree $b$. Then $S$ can be written as the union of the set $S_1$ of monomials of degree $b$ in the variables $x_1,x_2,....,x_{a-1}$ and a set $S_2$ disjoint from $S_1$. Observe $|S_1|=F(a-1,b)$. Now consider any element of $S_2$. Notice that such an element has a factor $x_a$. So if $p(\bar {{\rm x}}) \in S_2$, then there is a unique $\hat {p}(\bar {{\rm x}})$ such that $p(\bar{{\rm x}})=\hat{p}(\bar{{\rm x}})\cdot x_a$ and $\rm {deg(\hat{p}(\bar{{\rm x}}))}=b-1$. On the other hand, if $\hat {q}(x) \in k[x_1,x_2,....,x_a]$ and has degree $b-1$ then $(\hat {q}(x)) \cdot x_a \in S_2$. Therefore, there is a bijection from the set of monomials of degree $b-1$ in $k[x_1,x_2,....,x_a]$ to the set $S_2$. Consequently, $|S|=F(a-1,b)+F(a,b-1)$.
\end{proof}

Now we prove by induction that each element of the table is given by the following proposition.  Please be aware that the row count starts with $1$ but the column count starts with zero.  This is because the row count matches the number of variables used and the column count corresponds to the constant degree of the set of monomials being counted. 
\begin{proposition}\label{Pascal_table_combinatorial_formula}
$F(a,b)=\frac{(a-1+b)!}{(a-1)!b!}$, where $F(a,b)$, $a \geq 1, b \geq 0$, denotes the entry that lies in the $a^{\rm th}$ row and the $b^{\rm th}$ column of the Pascal table.  
\end{proposition} 
\begin{proof}
We have that $F(1,b)$ is the number of monomials of degree $b$ in a single variable. Since $x_1 ^{b}$ is the only monomial in $K[x_1]$ of degree $b$ then $F(1,b)=1$ for all $b\geq 0$. Also $F(a,0)=1$ for all $a\geq 1$ because in the ring $k[x_1,x_2,....,x_a]$ there is only one monomial of degree zero which is $x_1^{0}\cdot x_2^{0} \cdot ....\cdot x_a^{0}$.\\
{\bf Inductive Step}\\
Suppose $a>1$ and $b>0$. Then given that 
$$F(a-1,b)=\frac{(a-1+b-1)!}{(a-2)!b!}$$ and $$F(a,b-1) = \frac{(a-1+b-1)!}{(a-1)!(b-1)!}$$ 
we have 
\vspace{-2mm} 
\begin{eqnarray}
F(a,b)&=&F(a-1,b)+F(a,b-1) \nonumber \\
   & = &\frac{(a-1+b-1)!}{(a-2)!b!}+\frac{(a-1+b-1)!}{(a-1)!(b-1)!} \nonumber \\
   & = & \frac{(a-1+b)\cdot (a-2+b)!}{(a-1)!b!}\nonumber \\
   & = &\frac{(a-1+b)!}{(a-1)!b!}.\nonumber \
\end{eqnarray}
\end{proof}

Both meanings assigned to $F(a,b)$ are equivalent.  Thus, for example, we can say that by choosing $a=2$, we regard $F(2,b)$ as the value in the $2^{\rm nd}$ row and $b^{\rm th}$ column of the table or the number of monomials of degree $b$ that can be written with two distinct variables.  Also observe that the  proposition \ref{Pascal_table_combinatorial_formula} together with corollary \ref{corollary to HF for principal ideal} give a concrete formula for the Hilbert function of a principal ideal. So for $R=k[\bar {\rm x}]$ and $p \in R$ we can write
\begin{equation}\label{principal ideal Pascal}
{\rm HF}(R/I,b)= {\rm F}(a,b)-{\rm F}(a,b-{\rm deg}(p))=\frac{(a-1+b)!}{(a-1)!b!}-\frac{(a-1+b-\deg (p))!}{(a-1)!(b-\deg (p))!}.
\end{equation}

Proposition \ref{Pascal_table_recurrence_formula} is also valid for generating some rows of more general families of Hilbert functions.  We can prove it using either a counting argument or some homological algebra machinery.  We prefer the latter in order to avoid delicate counting procedures. Moreover, proposition \ref{Pascal_table_recurrence_formula} allows for an inductive construction of other expressions for computing values of the Pascal table. Let us illustrate this by expressing $F(a,b)$ in terms of the ascending factorial $[a]^{n}=a\cdot(a+1)\cdot(a+2)\cdot......\cdot(a+n-1)$
with the convention $[a]^{0}=1$.

\begin{proposition}\label{combinatorial_sum_formula}
The Hilbert function $F(a,b)$ defined as above it can be computed by either one of the following formulas
$$F(a,b)=\displaystyle\sum_{i=0}^{a-1} \frac{1}{i!} [b]^i
\quad \text{ or } \quad  
F(a,b)=\displaystyle\sum_{j=0}^{b} \frac{1}{j!} [a-1]^j.$$
\end{proposition}
\begin{proof}
To prove the first formula we observe that $F(1,b)=1$ for all $b \geq 0$ and this is precisely $F(1,b)=\displaystyle\sum_{i=0}^{1-1} \frac{1}{i!} [b]^i$.\\
We do induction on the first parameter of $F(a,b)$ namely $a \geq 2$.
Suppose 
\begin{center}
$F(a-1,b)=\displaystyle\sum_{i=0}^{a-2} \frac{1}{i!} [b]^i$.
\end{center}
Now we use the result that 
\[
F(a,b)=
  \begin{cases}
   F(a-1,b)+0,  & \text {for}\  b=0 \\
   
  F(a-1,b)+F(a,b-1), & \text {for}\  b>0  
  \end{cases}
\]
Observe that $F(a-1,0)=1$, for all $a\geq 1$.\\
Therefore,
\begin{eqnarray}
  F(a,b)  & = & \displaystyle F(a-1,b)+F(a,b-1) \nonumber \\
   & = & \displaystyle\sum_{i=0}^{a-2} \frac{1}{i!} [b]^i + \frac{(a-1+b-1)!}{(a-1)!(b-1)!}\nonumber \\
   & = & \displaystyle\sum_{i=0}^{a-2} \frac{1}{i!} [b]^i + \frac{1}{(a-1)!} \cdot (b\cdot (b+1)\cdot.......\cdot (b+a-2))\nonumber \\
   & = & \displaystyle\sum_{i=0}^{a-2} \frac{1}{i!} [b]^i +  \frac{1}{(a-1)!} \cdot [b]^{(a-1)}\nonumber\\
   & = &\displaystyle\sum_{i=0}^{a-1} \frac{1}{i!} [b]^i.\nonumber
\end{eqnarray}
The second formula follows immediately from the first formula since the left hand side is invariant when variables $a-1$ is interchanged with $b$. Therefore, we have that 
\begin{center}
$F(a,b)=\displaystyle\sum_{j=0}^{b} \frac{1}{j!} [a-1]^j$.
\end{center}
\end{proof}
As an example, take the graded module $k[x_1,x_2,x_3]$ then
$F(3,b)=[b]^0 + \frac{1}{1!}[b]^1 + \frac{1}{2!}[b]^2=1+b+ \frac{1}{2} (b^2 + b)$, where $b=0,1,2,....$ Now we proceed to create a more robust version to compute the Hilbert function of a quotient ring by introducing the meaning of the Hilbert function table.

\begin{definition}
A Hilbert function table associated to a quotient ring $k[x_1,x_2,...,x_d]/I$, where $I$ is a monomial ideal in $k[x_1,x_2,...,x_d]$ is an array whose entry indexed by $(a,b)$ is the value of ${\rm HF}(k[x_1,x_2,...,x_a]/I_a,b)$ , where $I_a$ is the ideal generated by the generators of $I$ that involve only the set of variables $\{x_1,x_2,...,x_a\}$.
\end{definition}

As a result of the above definition, the Pascal table is a Hilbert function table for graded modules of the form $k[\bar {\rm x}]$, where $\bar {\rm x}=(x_1,x_2,....,x_a)$ and $a \in  \mathbb{R}^{>0}$.

We can also observe that if $a \geq d$ then $I_a=I$. Moreover, the order of the variables $x_1,x_2,...,x_d$ will affect the Hilbert function table. In fact, two different Hilbert function tables for the same quotient ring need not have the same rows for $1 \leq a < d$. This is because altering the order of $x_1,x_2,...,x_d$ will alter the sequence of ideals $I_1,I_2,...I_{d-1}$. However, two Hilbert function tables for $k[x_1,x_2,...,x_a]/I$ will agree in rows $d$ and higher because $I_a=I$ for $a \geq d$.  After the $d^{\rm th}$ row, every new variable does not introduce a new monomial in the ideal. Therefore, producing the rows after the $d^{\rm th}$ row is a straightforward application of the following result.

\begin{theorem}\label{no_monomial_added_to_ideal}
Given $ {\rm HF}(j,b)$ the Hilbert Function of
 $k[x_1,x_2,....,x_d,x_{d+1},....,x_{d+j}]/I$, with  $j > 0 $,
where $I$ is a monomial ideal of the fixed set of variables $\{x_1,.....,x_d \}$ for $b \geq 0$ we have that $ {\rm HF}(j,b)= {\rm HF}(j-1,b) +  {\rm HF}(j,b-1)$.
\end{theorem}
\begin{proof}
For $j\geq1$, let $M_j=k[x_1,x_2,....x_d,x_{d+1},\dots,x_{d+j}]/I$ and let $z=x_{d+j}$.  We use the short exact sequence  
\begin{equation}\label{E: short exact sequence}
0\longrightarrow (0:z)(-1) \xrightarrow{\rm incl} M(-1) \longrightarrow M \longrightarrow  M/zM \longrightarrow 0
\end{equation}
found in \cite{villarreal2015monomial}.
In this short exact sequence let the term $M=M_j$.  Applying what are commonly known as the $2^{\rm nd}$ and $3^{\rm rd}$ isomorphism theorems or Proposition 2.1 in \cite{atiyah2018introduction}, 
\begin{align}
\begin{split}
zM_j &= z\left(k[x_1,x_2,....x_d,x_{d+1},\dots,x_{d+j}]/I\right)\nonumber\\ 
&\cong z\left(k[x_1,x_2,....x_d,x_{d+1},\dots,x_{d+j}]/(I \cap \langle z\rangle)\right)\nonumber\\
&\cong (z\,k[x_1,x_2,....x_d,x_{d+1},\dots,x_{d+j}] + I)/I\nonumber
\end{split}
\end{align}
thus,
\begin{align}
\begin{split}
M_j/zM_j&=( k[x_1,x_2,....x_d,x_{d+1},\dots,x_{d+j}]/I)/\left(z\, k[x_1,x_2,....x_d,x_{d+1},\dots,x_{d+j}]+I/I \right)\nonumber\\
&\cong  M_{j-1}=k[x_1,x_2,....x_d,x_{d+1},\dots,x_{d+j-1}]/I.\nonumber
\end{split}
\end{align} 
Since $z \notin I$ the only element $x \in M_j$ such that $zx=0$ is $x=0$.  In other words, the annihilator of multiplication by $z$ is zero.  This implies the short exact sequence, $0\longrightarrow (0:z)(-1) \xrightarrow{\rm incl} M_j(-1) \longrightarrow M_j \longrightarrow M_j/zM_j \longrightarrow 0$ and the corresponding alternating sum ${\rm HF}\{M_j/zM_j\}-{\rm HF}\{M_j\}+ {\rm HF}\{M_j(-1)\}=0$.
\end{proof}

\subsection{Motivating example:  the Stanley-Reisner Ring}

The Stanley-Reisner ring is a polynomial quotient ring assigned to a finite simplicial complex.  First, we must bring to the attention of the reader what is meant by a \emph{finite simplicial complex}.

\begin{definition}
A finite simplicial complex $\Delta$ consists of a finite set $V$ of vertices and a collection $\Delta$ of subsets of $V$ called faces such that
\begin{itemize}
\item{(i)}
If $u \in V$, then ${u} \in \Delta$.
\item{(ii)}
If $F \in \Delta$ and $G \subset F$, then $G \in \Delta$.
\end{itemize}
\end{definition}

{\bf Note:} The empty set is a face of every simplex.

Let $\Delta$ be a simplicial complex and let $F$ be a face of $\Delta$. Define the dimensions of F and $\Delta$ by ${\rm dimF}=|F|-1$ and ${\rm dim \Delta}={\rm sup} \{{\rm dim F} | F \in \Delta\}$ respectively. A face of dimension $q$ is called a q-face or a q-simplex.  Associate a distinct variable $x_i$ to each distinct vertex in the set $V$.  If $F$ is a face of $\Delta$ then the product of all corresponding $x_i$ is a square-free monomial associated with $F$.  This is due to the fact that at most one $q$--face can exist for a given $(q+1)$--set of vertices.  The Stanley-Reisner ring can be written in following form:
\begin{center}
 $K[x_1,x_2,.......,x_n]/I $,
\end{center}
where $I$ is an ideal of square free monomials ideal in the variables $x_1,x_2,.......,x_n$ corresponding to the non-face of $\Delta$.  For convenience let us denote the Stanley-Reisner ring associated with $\Delta$ by $k[\Delta]$. This is a standard construction the details of which can be found in page 5 of \cite{miller2004combinatorial}.

By definition, a simplicial complex $\Delta$ is a set theoretic construct but it is often the case we work with its  geometric realization.  That is associate with $\Delta$ a topological space that is a subspace of $\mathbb{R}^{\dim \Delta}$ and it is a union of simplices corresponding to the faces of $\Delta$.
Since $\Delta$ can be written as a disjoint union of its $i$-dimensional components $\Delta = \bigcup_{i=0}^{\dim \Delta}\Delta_i$ consequently the Stanley Reisner ring of  $\Delta$  admits a direct sum decomposition
\begin{center}
$k[\Delta]= \displaystyle \bigoplus _{i=0}^{\dim \Delta} k[\Delta_i]$
\end{center}
whose summands $k[\Delta_i]$ are vector spaces with a basis of monomials (not necessarily square-free) supported on the $i$-dimensional faces of $\Delta$.
\begin{example}
We illustrate how construction of the complex with Stanley--Reisner ring
\begin{center}
$M=k[x,\hat{x},y,z,w]/\langle x\hat{x},yzw \rangle$.
\end{center}
mirrors the generating of the corresponding Hilbert Function table by adding one variable at a time and including all relevant monomials in the ideal used in the quotient.
\end{example}

\begin{figure}[h!]
\centering
\includegraphics[width=0.2\textwidth]{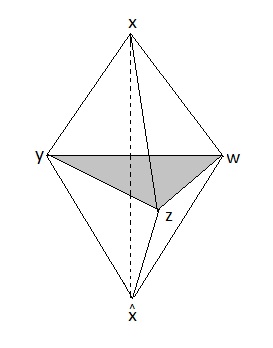}
\caption{4-vertices, 3-edges, 0-faces}
\end{figure}
\newpage
Start with the complex $C_0$ corresponding to the point $x$ we have the polynomial ring $k[x]$. Bringing the next variable $\hat{x}$, we have a new complex $C_1$ corresponding to the points $x,\hat{x}$. So we have $k[x,\hat{x}]/\langle x\hat{x} \rangle$. When the next variable $y$ shows up we have the complex $C_2$ corresponding to the points $x,\hat{x}, y$ and the edges $xy$ and $\hat{x}y$. By the same way, when $z$ shows up we have the complex $C_3$ corresponding to the points  $x,\hat{x}, y, z$, the edges $xy,xz,yz,y\hat{x},z\hat{x}$ and the faces $xyz$ and $y\hat{x}z$. To generate the table below we invoke\\ Theorem \ref{no_monomial_added_to_ideal}\\[12pt]

\begin{center}
\scalebox{1.2}{
\begin{tabular}{ l | l l l l l l l l l}
 ${\rm HF}\{k[x]\}$   & 1 & 1 & 1 & 1 & 1 & 1 & 1 & 1 & ... \\
\hline
${\rm HF} \{k[x,\hat{x}]/\langle x\hat{x} \rangle\}$  & 1 & 2 & 2 & 2 & 2 & 2 & 2 & 2 &...\\
\hline
${\rm HF} \{k[x,\hat{x}, y]/\langle x\hat{x} \rangle\}$  & 1 & 3 & 5 & 7 & 9 & 11 & 13 & 15 &...\\
\hline
${\rm HF} \{k[x,\hat{x}, y,z]/\langle x\hat{x} \rangle\}$  & 1 & 4 & 9 & 16 & 25 & 36 & 49 & 64 &...\\
\end{tabular}}
\end{center}
\vskip 0.5cm

Let $M_1=k[x,\hat{x}, y,z]/\langle x\hat{x} \rangle.$  Then by using the short exact sequence (s.e.s.) for $M$ we have

\begin{center}
\footnotesize{
\begin{tabular}{l*{10}{c}r}
{0} & $\longrightarrow$ & $b_3j$ &  $\xrightarrow {\text {inclusion}}$ & $b_2j$ & $\xrightarrow {{\text{multiply}}\ {\text{by}}\ w}$ & $b_1j$ &  $\longrightarrow$ & $b_0j$ & $\longrightarrow$ & $0$\\
{0} & $\longrightarrow$ & $(0:w)_{M}(-1)$ &  $\xrightarrow {\text {inclusion}}$ & ${M}(-1)$ & $\xrightarrow {{\text{multiply}}\ {\text{by}}\ w}$ & $M$ &  $\longrightarrow$ & $M/w {M} \cong {M_1}$ & $\longrightarrow$ & $0$\\
{0} & &{0} & &{0}  & & $\bf 1$ & & {1} & & {0}\\
{0} & &{0} & & {1} & & $\bf 5$ & &{4} & & {0} \\
{0} & &{0} & &{5}  & & $\bf 14$ & & {9} & & {0}\\
{0} & &{1} & &{14}  & & $\bf 29$ & & {16} & & {0}\\
{0} & &{4} & &{29}  & & $\bf 50$ & & {25} & & {0}\\
{0} & &{9} & &{50}  & & $\bf 77$ & & {36} & & {0}\\
{0} & &{16} & &{77}  & & $\bf 110$ & & {49} & & {0}\\
{0} & &{25} & &{110}  & & $\bf 149$ & & {64} & & {0}\\
{...} & & {...} & & {...} & &{...} & &{...} & & {...} \\
{...} & & {...} & & {...} & & {...} & &{...} & & {...} 
\end{tabular}
}
\end{center}

The justification for the values in the left most column is based on the annihilator 
\begin{center}
$(0:w)=\{q \in M:qw=0 \in M\}$
\end{center}
associated with the map which is multiplication by $w$. 
A basis for the b-graded component of the module $(0:w)$ is the following set: 
\begin{align}
B&=\left\{{\rm nonzero}\ p\in M\ {\rm of degree}\ b:yz|p\right\}\nonumber\\
&=\left\{(yz)(r):{\rm nonzero}\ r\in M_1\ {\rm with \ degree}\ b-2\ {\rm and}\ (x{\hat x})\nmid yzr\right\}\nonumber\\
&=\left\{(yz)(r):{\rm nonzero}\ r\in M_1 \ {\rm with \ degree}\ b-2\right\}.\nonumber
\end{align}
Thus, ${\rm HF}\{(0:w)\}=|B|={\rm HF}\{M_1(-2)\}$. Having accounted for all annihilator elements and using the fact that $b_2j=b_1j+b_3j-b_0j$
we find the Hilbert function for $M$. 

\newpage
\begin{example}
We are looking for the Hilbert function of the module 
\begin{center}
 $M=k[x,y,z]/\langle x^2yz^3,x^3z,y^2z^2 \rangle$.
\end{center}
\end{example}

By rearranging the variables in our example we have that 
$M=k[y,z,x]/\langle y^2z^2,x^2yz^3,x^3z \rangle$ and based on Theorem \ref{no_monomial_added_to_ideal} we have:
\begin{center}
\begin{tabular}{ l | l l l l l l l l l}
${\rm HF}\{k[y]\}$   & 1 & 1 & 1 & 1 & 1 & 1 & 1 & 1 &... \\
\hline
 ${\rm HF}\{k[y,z]/\langle y^2z^2\rangle\}={\rm HF}\{M_1\}$ & 1 & 2& 3 & 4 & 4 & 4 & 4 & 4 &...
 \end{tabular}
\end{center}

Therefore, based on the short exact sequence we have 
\begin{center}
\footnotesize{
\begin{tabular}{l*{10}{c}r}
{0}\hspace{-6mm} & $\longrightarrow$\hspace{-6mm} & $(0:x)(-1)$\hspace{-6mm} &  $\longrightarrow$\hspace{-4mm} & $M(-1)$\hspace{-4mm} & $\longrightarrow$\hspace{-6mm} & $M$\hspace{-4mm} &  $\longrightarrow$\hspace{-6mm} & $M_1$\hspace{-6mm} & $\longrightarrow$\hspace{-4mm} & $0$\\
{0} & &{0} & &{0}  & & ~~~$\bf 1$ & & ${1}=\{1\}$ & & {0}\\
{0} & &{0} & & {1} & & ~~~$\bf 3$ & &${2}=\{y,z\}$ & & {0} \\
{0} & &{0} & &{3}  & & ~~~$\bf 6$ & & ${3}=\{y^2,yz,z^2\}$ & & {0}\\
{0} & &{0} & &{6}  & & ~~~$\bf 10$ & & ${4}=\{y^3,y^2z,yz^2,z^3\}$ & & {0}\\
{0} & &${1}=\{x^2z\}$ & &{10}  & & ~~~$\bf 12$ & & ${4}=\{y^4,y^3z,yz^3,z^4\}$ & & {0}\\
{0} & &${2}=\{yx^2z,z^2x^2\}$ & &{12}  & & ~~~$\bf 13$ & & ${4}=\{y^5,yz^4,y^4z,z^5\}$ & & {0}\\
{0} & &${4}=\{1,y^2x^2z,yzx^2z,z^2x^2z\}$ & &{13}  & & ~~~$\bf 14$ & & {4} & & {0}\\
{...} & & {...} & & {...} & &{...} & &{...} & & {...} \\
{...} & & {...} & & {...} & & {...} & &{...} & & {...} 
\end{tabular}
}
\end{center}
In order to figure out the Hilbert function of the annihilator module we need to find all the non zero elements in $M$. Those elements should be either multiple of $xyz^3$ or $x^2z$. Therefore,
we cannot have a factor of $y$ and a factor of $y^2z$. In other words, there are no elements in $M_1$ that create $x^2yz^3$ and $x^3z$. However, there are elements in $M_1$ that create $y^2z^2$. By this way and using the fact that the alternating sum is zero we create the above table.
In this example, we can observe that the drawback is that computing the Hilbert function of the annihilator ideal would require counting. In the next examples, we illustrate basic approaches to avoid counting.
\subsection{Examples}
Now we use the basic results found earlier in this section to compute the Hilbert function of some key examples. These will provide the motivation for the techniques we develop in sections 3 and 4.  For our convenince, we group the examples based on the number of monomials generating the ideal used to produce the quotient ring.  
\subsubsection{The ideal used to produce the quotient polynomial ring is a principal ideal}
Consider $M=k[\bar{x}]/\langle u \rangle$ where $\deg u = d$. 
Using equation (\ref{principal ideal Pascal}) we obtain the following:
\begin{equation}\label{principal ideal case}
{\rm HF}(M,b)=
  \begin{cases}
   F(a,b),  & \text {for}\  0\leq b \leq d-1 \\
   F(a,b)-F(a,b - d),  &\text{for } b\geq d
  \end{cases}
\end{equation}  
This approach combined with the result in Proposition \ref{combinatorial_sum_formula} immediately yields
\begin{equation}\label{matrix}
{\rm HF}(M,b)=
  \begin{cases}
   F(a,b)=\displaystyle\sum_{j=0}^{b} \frac{1}{j!} [a-1]^j,  & \text {for}\  0\leq b \leq d-1 \\
   F(a,b)-F(a,b - d)=\displaystyle\sum_{j=b-(d-1)}^{b} \frac{1}{j!} [a-1]^j,  &\text{for } b\geq d
  \end{cases}
\end{equation}
Next, (\ref{matrix}) can be encoded as matrix multiplication using an infinite matrix and infinite column vectors corresponding to the right-hand side of the above equation.

\[\left( \begin{array}{ccccccccc}
\frac{1}{0!} & 0 & 0 & 0 & 0 & 0 &0 & 0 &... \\
\frac{1}{0!} & \frac{1}{1!} & 0 & 0 &0 & 0 & 0 & 0 &... \\
\frac{1}{0!} & \frac{1}{1!} & \frac{1}{2!} & 0 &0 & 0 & 0 & 0 &...\\
\vdots & \vdots &\vdots & \vdots & \vdots & \vdots & \vdots &\vdots & \vdots   \\
\frac{1}{0!} & \frac{1}{1!} & \frac{1}{2!} & \frac{1}{3!} & ... & \frac{1}{(d -1)!} & 0 &0 &...\\
0 & \frac{1}{1!} & \frac{1}{2!} & \frac{1}{3!} & ... & \frac{1}{(d-1)!} & \frac{1}{d!} &0 &...\\
0 & 0 & \frac{1}{2!} & \frac{1}{3!} & ... & \frac{1}{(d-1)!} & \frac{1}{d!} &\frac{1}{(d+1)!} &...\\
... &... & ... & ... & ... & ... & ... & ... & ...  \\
... & ... &... & ... & ... & ... & ... & ... & ...   \\
\end{array} \right) 
\cdot 
\left( \begin{array}{c}
\left[a-1\right]^{0}\\
\left[a-1\right]^1\\
\left[a-1\right]^2\\
\left[a-1\right]^3\\
\left[a-1\right]^4\\
\left[a-1\right]^5\\
\left[a-1\right]^6\\
...\\
...\\

\end{array} \right) = 
\left( \begin{array}{c}
{\bf {\rm HF}(M,0)}\\
{\bf {\rm HF}(M,1)}\\
{\bf {\rm HF}(M,2)}\\
{\bf \vdots}\\
{\bf {\rm HF}(M,d-1)}\\
{\bf {\rm HF}(M,d)}\\
{\bf {\rm HF}(M,d+1)}\\
{\bf ....}\\
{\bf ....}\\

\end{array} \right)   \hspace{1cm} .\]
In what follows, we concentrate our efforts in finding ways to compute the Hilbert function of a polynomial ring as finite sums and differences of the Pascal table row corresponding to the number of variables in our polynomial ring.  In each such case, one can produce a matrix multiplication approach similar to the above.  We'll leave this for the reader to try using the methods in section 4 as a starting point.  Here are two concrete examples to illustrate the above computations.
\vskip 0.5cm
\begin{example}
We are looking for the Hilbert function of the module 
 $M=k[x,y,z]/\langle xy^{2} \rangle$.
\end{example}
Equation (\ref{principal ideal case}) indicates the following recurrence relation for this quotient ring
\[
{\rm HF}(M,b)=
  \begin{cases}
   F(a,b),  & \text {for}\  0\leq b\leq 2 \\
   
  F(a,b)-F(a,b-3), & \text {for}\  b\geq 3  
  \end{cases}
\]
where the coefficients of $F(a,b)$ and $F(a,b-3)$ denote the first entry of the Pascal triangle. Therefore, the Hilbert function of the module $M=k[x,y,z]/\langle xy^{2} \rangle $ is expressed by the following sequence of numbers
\begin{center}
\begin{tabular}{ l l l l l l l l l}
 ${\rm HF}(M,b)$: {\bf 1} & {\bf 3} & {\bf 6} & {\bf 9} & {\bf 12} & {\bf 15} & {\bf 18} & {\bf ...} \\
 \end{tabular}
\end{center} 
 Finally, by (\ref{matrix}) we can rewrite the second part of the above function as 
\begin{center}
$ F(a,b)-F(a,b-3)=\displaystyle\sum_{j=b-2}^{b} \frac{1}{j!} [a-1]^j$,  for $b\geq 3$. 
\end{center}

An alternative way to express the Hilbert function of $M$ is given by the following way

\[\left( \begin{array}{cccccccc}
1 & 0 & 0 & 0 & 0 & 0 & 0 &... \\
1 & \frac{1}{1!} & 0 & 0 & 0 & 0 & 0 &... \\
1 & \frac{1}{1!} & \frac{1}{2!} & 0 & 0 & 0 & 0 &...\\
0 & \frac{1}{1!} & \frac{1}{2!} & \frac{1}{3!} & 0 & 0 & 0 &...\\
0 & 0 & \frac{1}{2!} & \frac{1}{3!} & \frac{1}{4!} & 0 & 0 &...\\
0 & 0 & 0 & \frac{1}{3!} & \frac{1}{4!} & \frac{1}{5!} & 0 &...\\
0 & 0 & 0 & 0 & \frac{1}{4!} & \frac{1}{5!} & \frac{1}{6!} &...\\
... & ... & ... & ... & ... & ... & ... & ...  \\
... & ... & ... & ... & ... & ... & ... & ...   \\
\end{array} \right) 
\cdot 
\left( \begin{array}{c}
\left[2\right]^{0}\\
\left[2\right]^{1}\\
\left[2\right]^{2}\\
\left[2\right]^{3}\\
\left[2\right]^{4}\\
\left[2\right]^{5}\\
\left[2\right]^{6}\\
...\\
...\\

\end{array} \right) = 
\left( \begin{array}{c}
{\bf 1}\\
{\bf 3}\\
{\bf 6}\\
{\bf 9}\\
{\bf 12}\\
{\bf 15}\\
{\bf 18}\\
{\bf ...}\\
{\bf ...}\\

\end{array} \right)   \hspace{1cm} .\]

\subsubsection{The ideal used in the quotient polynomial ring consists of two monomials}

Suppose $M=k[\bar{x}]/ \langle u,v \rangle$ with ${\rm deg}(u)=d_u$ and  ${\rm deg}(v)=d_v$. Again, the key to computing the ${\rm HF}(M,b)$ is finding a way to account exactly once for the monomials of degree $b$ belonging to the ideal $\langle u,v \rangle$. There are three possibilities which the reader may visualize as the Venn Diagram of two overlaying regions: one corresponding to $\langle u \rangle$ and the other to the corresponding $\langle v \rangle$. In fact, $q \in  \langle u \rangle$ and $q \in \langle v \rangle$ iff $q \in \langle{\rm lcm }(u,v) \rangle$. To see this, observe that $u \mid q$ and $v \mid q  \Leftrightarrow  {\rm lcm }(u,v) \mid q$.
Then by using the inclusion-exclusion principle we have
\begin{center} 
${\rm HF}(\langle u,v \rangle , b)= {\rm HF} ( \langle u \rangle ,b) + {\rm HF} ( \langle v \rangle ,b) - {\rm HF} ( {\rm lcm}(u,v),b).$ 
\end{center}
Since every ideal on the right-hand side is a principal ideal, we can apply lemma \ref{HF for principal ideal} and the ``rank--nullity''  reasoning from section 1 to get 
\[
{\rm HF}(M,b)=
  \begin{cases}
   F(a,b),  & \text {for }  0\leq b< d_{\min}  \\
   F(a,b) -F(a,b-d_{\min}), &\text{for } d_{\min} \leq b < d_{\max}\\
   F(a,b) -F(a,b-d_u) -F(a,b-d_v),   &\text{for } d_{\max} \leq b < d_{\rm lcm}\\
   F(a,b) -F(a,b-d_u) -F(a,b-d_v) +F(a,b- d_{\rm lcm}), &\text{for } b \geq d_{\rm lcm} \\
  \end{cases}
\]
where $d_{\rm lcm} = {\rm deg}({\rm lcm}(u, v))$, $d_{\min}=\min(d_u, d_v)$ and $d_{\max}=\max(d_u, d_v)$.

\begin{example}\label{two monomial ideal case}
We are looking for the Hilbert function of the module 
\begin{center}
 $M=k[x,y,z]/\langle x^{2}y,xz^{2} \rangle$.
\end{center}
\end{example}

Equation (\ref{principal ideal case}) indicates the following recurrence relation for this quotient ring
\[
{\rm HF}(M,b)=
  \begin{cases}
   F(a,b),  & \text {for}\  0\leq b\leq 2  \\
   
  F(a,b)-2F(a,b-3)+F(a,b-5), & \text {for}\  b\geq 3 
  \end{cases}
\]
Notice here that ${\rm lcm}(x^{2}y, xz^{2})=x^{2}yz^{2}$.
Therefore, the Hilbert function of the module $M=k[x,y,z]/\langle x^{2}y,xz^{2} \rangle$ is expressed by the following sequence of numbers
\begin{center}
\begin{tabular}{ l l l l l l l l l}
 ${\rm HF}(M,b)$:& {\bf 1} & {\bf 3} & {\bf 6} & {\bf 8} & {\bf 9} & {\bf 10} & {\bf 11} & {\bf ...} \\
 \end{tabular}
\end{center}

In the next section, we make full use of the \emph{Principle of Inclusion and Exclusion} to develop what we will call the \emph{lcm--lattice method} to handle any monomial ideal with a finite number of monomials.  Before moving to the next section, let us take advantage of this example to illustrate an alternative which accounts for the monomials of degree $b$ in the ideal only once.  In other words, the principle of inclusion-exclusion is a sequence of corrections for alternating over-counts and under-counts which corresponds to regions of the Venn diagram where two, three, four, etc... sets overlaps.  Our goal here is to partition the union of all sets in the Venn diagram into disjoint sets as to avoid alternating inclusions with exclusions.  This is accomplished by ordering our sets $E_1, E_2, E_3, \ldots$ then letting $F_1 = E_1, F_2=E_2 \setminus F_1, F_3=E_3 \setminus (F_1 \cup F_2), \ldots$.  This is an approach conceptually similar to the Gram-Schmidt process in linear algebra.

Let $u=x^2y$ and $v=xz^2$. Let also $E_1= \langle u \rangle$ and $E_2= \langle v \rangle$ then $F_1=E_1$ and $F_2=\{ \text{all monomials which are multiple of $v$ but not of $u$}\}$. Since $E_1$ and $E_2$ are graded modules then $F_1$ and $F_2$ will be graded sets. To illustrate this further, for degree $4$, $E_1$ and $E_2$ are disjoint so no monomials of degree $4$ need to be excluded from $F_2$. However, for degree $5$, for example $uz^2=vxy$. In this case, we want to count $x^2yz^2$ as a multiple of $u$ (i.e. belonging to $F_1$) but prevent it being counted as a multiple of $v$. Observe that
$\frac{ {\rm lcm}(x^2y,xz^2)}{xz^2}=\frac{x^2yz^2}{xz^2}=xy$
which is known as a syzygy. Thus, in example \ref{two monomial ideal case}, we can generate the table below 
\vspace{5mm}
\begin{center}
\begin{tabular}{l|l|l|l|l|l|l|l|l}
{ Degree k} & 0 & 1 & 2 & 3 & 4 &5 &6 & ...\\
\hline
{{\rm HF}\{k[x,y,z]\}} & 1 & 3 & 6 & 10 &15 &21 &28 & ... \\
\hline
{\bf $F_1$} & & & & $u$ & $ux, uy, uz,$ & $ux^2, uy^2, uz^2, $ & $ux^3, uy^3, uz^3, uxyz,$ & ...\\

& & & & &  &$uxy, uyz, uxz$&$ux^{2}y, uxy^{2}, ux^{2}z,$&...\\

& & & & & &  &$uy^{2}z, uxz^2, uyz^2$&...\\
\hline
{\bf $F_2$} &  &  &  & $v$ & $vx, vy, vz$ & $vx^2, vy^2, vz^2, $ & $vx^3, vy^3, vz^3, $ & ... \\

& & & & & & $vxz, vyz$ &$vx^{2}z ,vy^{2}z, vxz^2, vyz^2$&...\\

\hline
{\bf $|F_1|$} &  &  &  & 1 &3 &6 &10 & ... \\
\hline
{\bf $|F_2|$} &  &  &  & 1 &3 &5 &7 & ... \\
\hline
{\bf {\rm G(a,b)} } & 1 & 3 & 6 & 8 &9 &10 &11 & ... \\
\end{tabular}
\end{center}
\vspace{8mm}
where $G(a,b)={\rm HF} \{M\}={\rm HF}\{k[x,y,z]\}-|F_1|-|F_2|$.


\newpage
\section{LCM--Lattice Method}

As discussed at the end of the previous section, the challenge remains to find the Hilbert function of a monomial ideal with more than one monomial generators. In this section, first we start with some basic theory and then use the well known \emph{Principle of Inclusion and Exclusion} (which the reader will find in the standard reference \cite{van2001course}) to validate the method developed.
\begin{proposition}
$\langle u \rangle \cap \langle v \rangle=\langle {\rm lcm}(u,v) \rangle $ 
\end{proposition}
\begin{proof}
$p \in \langle u \rangle \cap \langle v \rangle \Leftrightarrow u\mid p$ and $v \mid p \Leftrightarrow {\rm lcm }(u,v) \mid p$
\end{proof}

\begin{corollary}
$\langle p_1 \rangle \cap \langle p_2 \rangle  \cap \langle p_3 \rangle \cap ...\cap  \langle p_r \rangle =\langle {\rm lcm}(p_1,p_2,p_3,...,p_r) \rangle $ 
\end{corollary}
\begin{proof}(By Induction)
\begin{itemize}
\item
The above proposition is the above case.
\item
Suppose $\displaystyle \bigcap _{i=1}^{r-1} \langle p_i \rangle=\langle {\rm lcm}(p_1,p_2,p_3,...,p_{r-1}) \rangle $ 
then 
\begin{center}
$\displaystyle \bigcap _{i=1}^{r} \langle p_i \rangle=\displaystyle \bigcap _{i=1}^{r-1} \langle p_i \rangle \cap  \langle p_r \rangle=  \langle {\rm lcm}({\rm lcm }(p_1,p_2,p_3,...,p_{r-1}),p_r) \rangle= \langle {\rm lcm }(p_1,p_2,p_3,...,p_r) \rangle.$
\end{center}
\end{itemize}
\end{proof}

Further use of inclusion-exclusion; this time with $n$ monomials we get 
\begin{eqnarray}
{\rm HF}( \langle p_1,p_2,p_3,...,p_n \rangle ,b)&=& |\{ \text{monomials of degree }b\text{  in }\langle  p_1,p_2,p_3,...,p_n \rangle \}| \nonumber \\
 & = &|\{ \text {monomials of degree } b \text{ in } \langle p_1 \rangle \text{ OR } \langle p_2 \rangle \text{ OR}...\text{ OR } \langle p_n \rangle \}| \nonumber \\
&=& \displaystyle\sum_{1 \leq j_1 \leq n} |\langle p_{j_1} \rangle |  -  \displaystyle\sum_{1 \leq j_1 <j_2 \leq n} |\langle {\rm lcm}(p_{j_1},p_{j_2})\rangle | \nonumber\\
& & + \displaystyle\sum_{1 \leq j_1 <j_2 <j_3 \leq n} |\langle (p_{j_1},p_{j_2},p_{j_3}) \rangle|\nonumber\\
& & +............................................................. \nonumber\\
& & + (-1)^{r-1} \cdot \displaystyle\sum_{1 \leq j_1 <j_2<...<j_r \leq n} |\langle {\rm lcm}(p_{j_1},p_{j_2},...,p_{j_r})\rangle | \nonumber\\
& & + (-1)^{n-1} \cdot  |\langle {\rm lcm}(p_{1},p_{2},...,p_{n})\rangle | \nonumber\\  
&=&  \displaystyle\sum_{r=1}^{n}\left((-1)^{r-1} \cdot \displaystyle\sum_ {1 \leq j_1 <j_2<...<j_r \leq n} | \langle {\rm lcm} (p_{j_1},p_{j_2},...,p_{j_r})\rangle|\right ).\nonumber \ 
\end{eqnarray}
 Assigning $d_{{j_1}{j_2}{j_3}...{j_r}}= {\rm deg} ( {\rm lcm}(p_1,p_2,p_3,...,p_r))$, where $1 \leq r \leq n$ and  $1 \leq j_1 < j_2 <j_3<...<j_r \leq n$. To facilitate expressing the Hilbert function let's expand $F(a,b)=0$ if $ b<0$.  Then
$${\rm HF}( \langle p_1,p_2,p_3,...,p_n \rangle ,b) =\displaystyle \sum_{r=1}^{n}\left((-1)^{n-1} \cdot \displaystyle\sum_ {1 \leq j_1 <j_2<...<j_r \leq n} F(a,b-d_{{j_1}{j_2}{j_3}...{j_r}})\right)$$
\begin{center}
and
\end{center}
$${\rm HF}(k[ \bar{x}] / \langle p_1,p_2,p_3,...,p_n \rangle ,b)=
F(a,b)- \displaystyle \sum_{r=1}^{n}\left((-1)^{n-1} \cdot \displaystyle\sum_ {1 \leq j_1 <j_2<...<j_r \leq n} F(a,b-d_{{j_1}{j_2}{j_3}...{j_r}})\right).$$

The lcm--lattice method described below is based on the above argument.
The starting point of building up the lcm--lattice is what we call layer 1.  Layer 1 is a  row containing all the monomials of the given ideal. Finding the lcm of all the pairs we create the $2^{\rm nd}$ layer. Next, we find the lcm of all the triples in layer 1 and we call this layer 3. Following the same pattern, we create as many layers as the number of monomials in the given ideal. The last layer will contain the lcm of all the monomials given in the ideal.
If the ideal $I$ contains $n$ monomials then the number of monomials in the lcm-lattice in layers 1, 2, 3, ..., n will be $\dbinom{n}{1},\,\dbinom{n}{2},\,\dbinom{n}{3},\,\ldots, \dbinom{n}{n}$ correspondingly.  These values are those found in the $n^{\rm th}$ row of the Pascal triangle to the right and including $\dbinom{n}{1}$. The following examples give a nice view of the above description.
\begin{example}
Find the Hilbert function of 
$M=R/\langle x^2,y^3 \rangle$ where $R=k[x,y,z]$.
\end{example}
In the case that we have two monomials in the ideal, the lcm lattice is simple.
Start by building up the lcm lattice.
Layer 1 is called the row that has all the monomials of the ideal. 
Afterwards, we take the lcm of the two monomials and we have the following 
\begin{center}
\begin{tabular}{l c l }
$x^2$&$y^3$& layer 1\\
~~~~~$x^2y^3$&& layer 2
\end{tabular}
\end{center}

According now to the above lcm lattice, we are left with a lattice of monomials on which we use inclusion - exclusion at each row to produce the alternating sum that computed the Hilbert function
\begin{center}
\scalebox{0.85}{
\begin{tabular}{ l | l | l | l  }
${\rm HF}\{R\}$ & layer 1(-) & layer 2($+$) & ${\rm HF} \{M\}$\\
\hline
 ~1 & ~~~~0  ~~~~~0 &  ~~~0 & ~~~ $\bf 1$ \\
\hline
 ~3 & ~~~~0  ~~~~~0 & ~~~0 & ~~~ $\bf 3$\\
\hline
 ~6 & ~~~-1  ~~~~~0 &  ~~~0 & ~~~ $\bf 5$ \\
\hline
 10 & ~~~-3  ~~~~-1 &   ~~~0 & ~~~ $\bf 6$ \\
\hline
 15 & ~~~-6  ~~~~-3 &  ~~~0 & ~~~ $\bf 6$ \\
\hline
 21 & ~~~-10  ~~~-6 &  ~~~1 & ~~~ $\bf 6$ \\
\hline
 28 & ~~~-15  ~~~-10 &   ~~~3 & ~~~ $\bf 6$ \\
\hline
 36 & ~~~-21  ~~~-15 &  ~~~6 & ~~~ $\bf 6$ \\
\hline
 45 & ~~~-28  ~~~-21 & ~~~10 & ~~~ $\bf 6$ \\
\hline
 55 & ~~~-36  ~~~-28 & ~~~15 & ~~~ $\bf 6$ \\
\hline
 ~.. & ~~~~..  ~~~~~.. & ~~~.. & ~~~~.. \\
\end{tabular}}
\end{center}


\begin{example}
Given the quotient ring $M=k[x,y,z]/\langle  xz,yz,x^2y \rangle$, find its Hilbert function.
\end{example}
Start by building up the lcm lattice.

\begin{center}
\begin{tabular}{c c c c}
${\bf x^2y}$&$xz$&$yz$&layer 1\\
${\bf{x^2yz}}$&$x^2yz$&${\bf xyz}$&layer 2\\
$$&${\bf{x^2yz}}$&&layer 3\\
\end{tabular}
\end{center}
So we are left with the above lattice of monomials on which we use inclusion - exclusion at each row to produce the alternating sum that computed the Hilbert function.
Let $R=k[x,y,z]$. Since we observe that there are monomials of the same degree in adjacent rows of the lcm-lattice lattice, we exclude these pairs of monomials from the alternating sum in our table. The cancellation is because the net contribution of such a pair to the alternating sum is zero. The monomials that are canceled are displayed in bold-faced.

\begin{center}
\begin{tabular}{ c|c|c|c }
${\rm HF}\{R\}$ & layer 1(-) & layer 2($+$) & ${\rm HF}\{M\}$\\
\hline
 1 &0  &0 & $\bf 1$ \\
\hline
 3 & 0  &0 & $\bf 3$\\
\hline
 6 &-2  &  0 & $\bf 4$ \\
\hline
 10 & -6  &  0 &  $\bf 4$ \\
\hline
 15 & -12  & 1 & $\bf 4$ \\
\hline
 21 & -20  & 3 &  $\bf 4$ \\
\hline
 28 & -30  & 6 &  $\bf 4$ \\
\hline
 36 & -42  & 10 & $\bf 4$ \\
\hline
 45 & -56  &15 & $\bf 4$ \\
\hline
 55 & -72  &21 & $\bf 4$ \\
\hline
 ..&..  & .. & .. \\
\hline
 .. & .. & ..& .. 
\end{tabular}
\end{center}

\newpage
\begin{example}
Compute the Hilbert function of the module 
\begin{center}
 $M=k[x,y,z]/\langle  x^2y^3z,xz^3,xy^4z,x^2z^2 \rangle$.
\end{center}
As before, for the sake of simplicity, we will let $R=k[x,y,z]$.
\end{example}
Start by building up the lcm lattice and we have
\begin{center}
\begin{tabular}{c c c c c c c}
$ $&$x^2y^3z$&$xz^3$&$xy^4z$&$x^2z^2$&$$& layer 1\\
${\bf x^2y^3z^3}$&$x^2y^4z$&$x^2y^3z^2$&$xy^4z^3$&$x^2z^3$&${\bf x^2y^4z^2}$& layer 2\\
$ $&${\bf x^2y^4z^3}$&${\bf x^2y^3z^3}$&$x^2y^4z^3$&${\bf x^2y^4z^2}$&$$& layer 3\\
$$&$$&$$&${\bf x^2y^4z^3}$&$$&$$&layer 4\\
\end{tabular}
\end{center}
Again, we typeset in bold face the monomials that are canceled and obtain the following table
\begin{center}
\begin{tabular}{ l | l | l | l | l }
${\rm HF}\{R\}$ & ~~~layer 1& ~~~~layer 2 &layer 3&${\rm HF}\{M\}$\\
\hline
 ~1 & ~~~~0  ~~~~~0 & ~~~0  ~~~0 ~~~0 & ~~~ 0&~~~ $\bf 1$ \\
\hline
 ~3 & ~~~~0  ~~~~~0 & ~~~0  ~~~0 ~~~0 &~~~ 0& ~~~ $\bf 3$\\
\hline
 ~6 & ~~~~0  ~~~~~0 & ~~~0  ~~~0 ~~~0 &~~~ 0& ~~~ $\bf 6$ \\
\hline
 10 & ~~~~0  ~~~~~0 & ~~~0  ~~~0 ~~~0 &~~~ 0& ~~ $\bf 10$ \\
\hline
 15 & ~~~-2  ~~~~~0 & ~~~0  ~~~0 ~~~0&~~~ 0& ~~~$\bf 13$ \\
\hline
 21 & ~~~-6  ~~~~~0 & ~~~1  ~~~0 ~~~0 &~~~ 0& ~~ $\bf 16$ \\
\hline
 28 & ~~-12  ~~~-2 & ~~~3  ~~~0 ~~~0 &~~~ 0& ~~ $\bf 17$ \\
\hline
 36 & ~~-20  ~~~-6 & ~~~6  ~~~2 ~~~0&~~~ 0& ~~ $\bf 18$ \\
\hline
 45 & ~~-30  ~~-12 & ~~10  ~~~6 ~~1 &~~~ 0& ~~ $\bf 20$ \\
\hline
 55 & ~~-42  ~~-20 & ~~15  ~~12 ~~3 &~~~ -1& ~~ $\bf 22$ \\
\hline
 ~.. & ~~~~..  ~~~~~.. & ~~~..~~~.. ~~~..& ~~~~..& ~~~~.. \\
\hline
~.. & ~~~~..  ~~~~~.. & ~~~..~~~..~~~.. & ~~~~..& ~~~~.. \\
\end{tabular}
\end{center}


\section{The Syzygy Method}

In this section we extend the second approach to the example \ref{two monomial ideal case} to handle ideals with finitely many monomials as generators.  When implemented as a recursive algorithm this method will break down a Hilbert function computation into a sum--difference expression of Hilbert functions all of which involve principal ideals.  The computation is finished by invoking Corollary \ref{principal ideal case}.  Unlike the lcm--method, the principal ideals used will be generated by always taking syzygys of pairs of monomials (we never consider three or more of the given monomials in a computational step).  The key recursive step is given by the following theorem. 
\begin{theorem}{\rm (Syzygy method)}\label{Syzygy method}
Let $M=k[\bar  {\rm x}]/ I $,
where $I=\langle p_1, p_2, p_3,..., p_r \rangle$, a monomial ideal generated by $p_1,p_2,p_3,...,p_r \in k[\bar  {\rm x}]$.  Then, using the notation  $d_j={\rm deg}(p_j)$ and $m_{ij}=\frac{{\rm lcm}(p_i,p_j)}{p_j} \in k[\bar  {\rm x}]$ with $i<j$, we have
\begin{center}
\small{${\rm HF}(M,t)=F(a,t)-F(a,t-d_1)-\displaystyle\sum_{j=2}^{r} {\rm HF}(k[\bar  {\rm x}]/\langle m_{1j}, m_{2j}, m_{3j},..., m_{(j-1)j}\rangle, t-d_{j})$}. 
\end{center}
\end{theorem}

\begin{proof}
(By Induction)

\begin{itemize}
\item 
Base case $r=1$ then this hold by the corollary 1.0.1.
\item
Suppose $r>1$ and
\vspace{-1mm}
\begin{eqnarray}
{\rm HF}(k[\bar  {\rm x}]/\langle p_1, p_2, p_3,..., p_{r-1}\rangle,t)&=&F(a,t)-F(a,t-d_1)\nonumber\\  
& &-\displaystyle\sum_{j=2}^{r-1} {\rm HF}(k[\bar  {\rm x}]/\langle m_{1j}, m_{2j}, m_{3j},..., m_{(j-1)j}\rangle, t-d_{j}).\nonumber
\end{eqnarray}
\item
We show that 
\vspace{-1mm}
\begin{eqnarray}
{\rm HF}(k[\bar  {\rm x}]/\langle p_1, p_2, p_3,..., p_{r} \rangle,t)&=&{\rm HF}(k[\bar  {\rm x}]/\langle p_1, p_2, p_3,..., p_{r-1} \rangle,t)\nonumber\\
& &- {\rm HF}(k[\bar  {\rm x}]/\langle m_{1r}, m_{2r}, m_{3r},..., m_{(r-1)r}\rangle, t-d_{r}).\nonumber
\end{eqnarray}
A monomial $q \in k[\bar  {\rm x}]$, of degree $t$ represent a nonzero element in $k[\bar  {\rm x}]/\langle p_1, p_2, p_3,..., p_{r-1} \rangle$ and is zero in $k[\bar  {\rm x}]/\langle p_1, p_2, p_3,..., p_{r} \rangle$ if and only if 
$p_i \nmid q$ for all  $1\leq i<r$ and $p_r|q$. If we call the set of all such monomials $\Gamma (t)$ then we have that
\begin{center}
\small{${\rm HF}(k[\bar  {\rm x}]/\langle p_1, p_2, p_3,..., p_{r} \rangle,t)={\rm HF}(k[\bar  {\rm x}]/\langle p_1, p_2, p_3,..., p_{r-1} \rangle,t)-|\Gamma (t)|$}.
\end{center}
A monomial  $q \in k[\bar  {\rm x}]$ satisfies  $q \in \Gamma (t)\Leftrightarrow q=a \cdot p_r$, where $a$ is a monomial in $k[\bar  {\rm x}]$ of degree $t-d_r$ and $p_i \nmid a \cdot p_r$ for all $1\leq i<r $. This is equivalent to $m_{ir} \nmid a$ for all $1\leq i<r $.  
Since $a$ is a monomial we have that,
\vspace{-5mm}
\begin{align}
a \notin \langle m_{ir} \rangle, &\text{ for all } 1\leq i \leq r-1 \nonumber\\
&\Leftrightarrow a \notin  \langle  m_{1r}, m_{2r}, m_{3r},..., m_{(r-1)r}\rangle \nonumber\\ 
&\Leftrightarrow a \in k[\bar  {\rm x}]/\langle m_{1r}, m_{2r}, m_{3r},..., m_{(r-1)r}\rangle.\nonumber
\end{align}
Finally, to finish the proof and establish that 
\begin{center}
{$|\Gamma (t)|= {\rm HF}(k[\bar  {\rm x}]/\langle m_{1r}, m_{2r}, m_{3r},..., m_{(r-1)r}\rangle, t-d_{r})$}
\end{center}
we only need to observe that $a$ is uniquely determined by $q \in\Gamma (t)$ and every
$a \in k[\bar  {\rm x}]/\langle m_{1r}, m_{2r}, m_{3r},..., m_{(r-1)r}\rangle$ uniquely determines a monomial $q$. 
\end{itemize}
\vspace{-5mm}
\end{proof}

Both the lcm-lattice-method and the Syzygy method produce similar formulas for computing the Hilbert function.  Next we apply the Syzygy method to establish that the lcm-lattice method holds for a monomial ideal with three monomials. The reader should observe that this will confirm of that result without the use of inclusion-exclusion.
Consider $I$ generated by three (not necessarily distinct) monomials $p_1,p_2,p_3$ with degrees $d_1,d_2,d_3$ respectively. We need to show that 
\begin{eqnarray}
 HF\{k[\bar  {\rm x}]/\langle p_1,p_2,p_3 \rangle\}& = &F(a,t)-F(a,t-{\rm deg}(p_1))\nonumber \\    
    &  & -F(a,t-{\rm deg}(p_2))-F(a,t-{\rm deg}(p_3))\nonumber \\
    &  & +F(a,t-{\rm deg}({\rm lcm}(p_1,p_2)))+F(a,t-{\rm deg}({\rm lcm}(p_2,p_3)))\nonumber\\
    &  & +F(a,t-{\rm deg}({\rm lcm}(p_1,p_3)))-F(a,t-{\rm deg}({\rm lcm}(p_1,p_2,p_3))).\nonumber
\end{eqnarray}
By the syzygy method, we obtain the following equality which we call the syzygy equality
\begin{eqnarray}
 HF\{k[\bar  {\rm x}]/\langle p_1,p_2,p_3 \rangle\}&=&F(a,t)-F(a,t-d_1)-HF\{k[\bar  {\rm x}]/\langle m_{12} \rangle(-d_2)\}\nonumber\\ & &-HF\{k[\bar  {\rm x}]/\langle m_{13},m_{23} \rangle(-d_3)\}\nonumber
\end{eqnarray}
Applying the syzygy method to the third and fourth summands on the right hand side we have 
\begin{eqnarray}
 HF\{k[\bar  {\rm x}]/\langle m_{12} \rangle(-d_2)\}& = &F(a,t-d_2)-F(a,t-d_2-{\rm deg}(m_{12}))\nonumber \\    
    & = & F(a,t-d_2)-F(a,t-{\rm deg}({\rm lcm}(p_1,p_2)))\nonumber 
\end{eqnarray}
and
\begin{eqnarray}
 HF\{k[\bar  {\rm x}]/\langle m_{13},m_{23} \rangle(-d_3)\}& = &F(a,t-d_3)-F(a,t-d_3-{\rm deg}(m_{13}))\nonumber \\    
    &  &- HF\left\{k[\bar  {\rm x}]/\left\langle {\rm lcm}(m_{13},m_{23})m_{23}^{-1} \right\rangle\left(-d_3-{\rm deg}(m_{23})\right)\right\}. \nonumber 
\end{eqnarray}
and
\vspace{-5mm}
\begin{eqnarray}
&HF&\{k[\bar{\rm x}]/\langle\frac{{\rm lcm}(m_{13},m_{23})}{m_{23}} \rangle(-d_3-{\rm deg}(m_{23}))\}=\nonumber\\
&= & F(a,t-d_3-{\rm deg}(m_{23})-F \left(a,t-d_3-{\rm deg}(m_{23})-{\rm deg}\left(\frac{{\rm lcm}(m_{13},m_{23}}{m_{23}}\right)\right)\nonumber \\    
    &= &F(a,t-{\rm deg}({\rm lcm}(p_2,p_3)))-F (a,t-d_3-{\rm deg}{\rm lcm}(m_{13},m_{23})) \nonumber \\
    &= &F(a,t-{\rm deg}({\rm lcm}(p_2,p_3)))-F(a,t-(d_3+{\rm deg}{\rm lcm}(m_{13},m_{23})) \nonumber\\
    &= &F(a,t-{\rm deg}({\rm lcm}(p_2,p_3)))-F\left(a,t-({\rm deg}(p_3)+{\rm deg} \left ({\rm lcm}\left(\frac{{\rm lcm}(p_1,p_3)}{p_3},\frac{{\rm lcm}(p_2,p_3)}{p_3}\right)\right)\right) \nonumber\\
    &= &F(a,t-{\rm deg}({\rm lcm}(p_2,p_3)))-F\left(a,t-\left({\rm deg}(p_3)+{\rm deg}\left(\frac{{\rm lcm}(p_1,p_2,p_3)}{p_3}\right)\right)\right) \nonumber\\
    &= &F(a,t-{\rm deg}({\rm lcm}(p_2,p_3)))-F(a,t-{\rm deg}({\rm lcm}(p_1,p_2,p_3)))). \nonumber
\end{eqnarray}
Back-substituting the iterated results of the Syzygy method into the syzygy equality produces the same alternating sum as the lcm-method.  Thus, we proved that the lcm-lattice method is valid.
\newpage

The following examples are based on the Syzygy method.
\begin{example}
Find the Hilbert function of 
$M=R/\langle x^2,y^3 \rangle$, where 
$R=k[x,y,z]$.
\end{example}
We will only need the syzygy 
$m_{12}=\frac{{\rm lcm}(x^2,y^3)}{y^3}=\frac{x^2y^3}{y^3}=x^2$. 

Computing the Hilbert function in this case requires only one use Theorem \ref{Syzygy method}, which yields the following:   
\begin{eqnarray}\label{syzygy equality}
{\rm HF}\{M\}&=&{\rm HF}\{R\}-{\rm HF}\{R({\rm -deg}(x^2))\}-{\rm HF}\{R/{\langle m_{12}  \rangle}  ({\rm -deg}(y^3))\}\nonumber\\ &=&{\rm HF}\{R\}-{\rm HF}\{R(-2)\}-{\rm HF}\{R/{\langle x^2 \rangle} (-3)\}.
\end{eqnarray}
Based on the Corollary \ref{HF for principal ideal}, we see that the last term in (\ref{syzygy equality}) it is shifted by 3, so we have
\begin{eqnarray}\label{to sub into syzygy equality}
{\rm HF}\{R/{\langle x^2 \rangle} (-3)\}&=&{\rm HF}\{R(-3)\}-{\rm HF}\{R ({\rm -deg}(x^2))(-3)\}\nonumber\\ &=&{\rm HF}\{R(-3)\}-{\rm HF}\{R(-5)\}.
\end{eqnarray}
Therefore, by substituting equation (\ref{to sub into syzygy equality}) into equation (\ref{syzygy equality}), we have
 \begin{eqnarray}
{\rm HF}\{M\}&=&{\rm HF}\{R\}-{\rm HF}\{R(-2)\}-[{\rm HF}\{R(-3)\}-{\rm HF}\{R(-5)\}]\nonumber\\ &=&{\rm HF}\{R\}-{\rm HF}\{R(-2)\}-{\rm HF}\{R(-3)\}+{\rm HF}\{R(-5)\}.\nonumber
\end{eqnarray}
The Hilbert function of $M$ is presented in the last column of the following row-generating table
\begin{center}
\scalebox{1.1}{
\begin{tabular}{ l | l | l | l | l}
${\rm HF}\{R\}$ & ${\rm -HF}\{R(-2)\}$ & {\rm -HF}\{R(-3)\} & {\rm HF}\{R(-5)\} & ${\rm HF}\{M\}$\\
\hline
 ~1 & ~~~~~0 & ~~~~~0 & ~~~~~~~~0 &~~~ $\bf 1$ \\
\hline
 ~3 & ~~~~~0 & ~~~~~0 & ~~~~~~~~0 &~~~ $\bf 3$\\
\hline
 ~6 & ~~~~-1 & ~~~~~0 & ~~~~~~~~0 &~~~ $\bf 5$ \\
\hline
 10 & ~~~~-3 & ~~~~-1 & ~~~~~~~~0 &~~~ $\bf 6$ \\
\hline
 15 & ~~~~-6 & ~~~~-3 & ~~~~~~~~0 &~~~ $\bf 6$ \\
\hline
 21 & ~~~-10 & ~~~~-6 & ~~~~~~~~1 &~~~ $\bf 6$ \\
\hline
 28 & ~~~-15 & ~~~-10 & ~~~~~~~~3 &~~~ $\bf 6$ \\
\hline
 36 & ~~~-21 & ~~~-15 & ~~~~~~~~6 &~~~ $\bf 6$ \\
\hline
 45 & ~~~-28 & ~~~-21 & ~~~~~~~10 &~~~ $\bf 6$ \\
\hline
 55 & ~~~-36 & ~~~-28 & ~~~~~~~15 &~~~ $\bf 6$ \\
\hline
 66 & ~~~-45 & ~~~-36 & ~~~~~~~21 &~~~ $\bf 6$ \\
\hline
 78 & ~~~-55 & ~~~-45 & ~~~~~~~28 &~~~ $\bf 6$ \\
\hline
 ~.. & ~~~~.. & ~~~~~.. & ~~~~~~~~..&~~~~.. \\
\hline
~.. & ~~~~.. & ~~~~~.. & ~~~~~~~~..&~~~~.. \\
\end{tabular}}
\end{center}

\vspace{8mm}
In the next example, we apply the Syzygy method to quotient rings whose monomial ideal consists of more than two monomials.

\begin{example}
Denote by  $R=k[x,y,z]$.  Find the Hilbert function of the quotient ring $M=R/\langle  xz,yz,x^2y \rangle$.  Observe that in this example we have the following syzygies,
\end{example}
\begin{center}
$m_{12}=\frac{{\rm lcm}(xz,yz)}{yz}=\frac{xyz}{yz}=x$,\\
$m_{13}=\frac{{\rm lcm}(xz,x^2y)}{x^2y}=\frac{x^2yz}{x^2y}=z$,\\
$m_{23}=\frac{{\rm lcm}(yz,x^2y)}{x^2y}=\frac{x^2yz}{x^2y}=z$.
\end{center}
\vspace{-2mm}
Using now the syzygy method we can express the Hilbert function of $M$ as follows
\begin{eqnarray}\label{syzygy equality2}
{\rm HF}\{M\}&=&{\rm HF}\{R\}-{\rm HF}\{R({\rm -deg}(xz))\}-{\rm HF}\{R/{\langle m_{12}  \rangle}  ({\rm -deg}(yz))\}\nonumber\\
& & -{\rm HF}\{R/{\langle m_{13},m_{23}  \rangle} ({\rm -deg}(x^2y))\}\nonumber\\ &=&{\rm HF}\{R\}-{\rm HF}\{R(-2)\}-{\rm HF}\{R/{\langle x \rangle} (-2)\}-{\rm HF}\{R/{\langle z  \rangle} (-3)\}.
\end{eqnarray}
The last two terms of \ref{syzygy equality2} equal to the second row of the Pascal table but shifted since
${\rm HF}\{R/{\langle x \rangle} (-2)\}\cong {\rm HF}\{k[y,z](-2)\}$\quad and \quad
${\rm HF}\{R/{\langle z \rangle} (-3)\}\cong {\rm HF}\{k[x,y](-3)\}$.

Thus we obtain the Hilbert function of $M$  shown in the last column of the table below.
\begin{center}
\scalebox{0.85}{
\begin{tabular}{ c|c|c|c|c}
${\rm HF}\{R\}$ & ${\rm -HF}\{R(-2)\}$ & ${\rm -HF}\{R/{\langle x \rangle}  (-2)\}$ & ${\rm -HF}\{R/{\langle z  \rangle}  (-3)\}$ & ${\rm HF}\{M\}$\\
\hline
 ~1 & 0 & 0 & 0 &$\bf 1$ \\
\hline
 ~3 & 0 & 0 & 0 & $\bf 3$\\
\hline
 ~6 & -1 & -1 & 0 & $\bf 4$ \\
\hline
 10 & -3 & -2 &-1 & $\bf 4$ \\
\hline
 15 &-6 &-3 &-2 & $\bf 4$ \\
\hline
 21 & -10 & -4 & -3 & $\bf 4$ \\
\hline
 28 & -15 &-5 &-4 & $\bf 4$ \\
\hline
 36 & -21 & -6 &-5 & $\bf 4$ \\
\hline
 45 & -28 &-7 & -6 & $\bf 4$ \\
\hline
 55 & -36 & -8 & -7 &$\bf 4$ \\
\hline
 66 & -45 & -9 & -8 &$\bf 4$ \\
\hline
 78 & -55 & -10 & -9 & $\bf 4$ \\
\hline
 .. & .. & .. &..&.. \\
\end{tabular}}
\end{center}

\begin{example}
Compute the Hilbert function of the module 
\begin{center}
 $M=k[x,y,z]/\langle  x^2z^2,xz^3,xy^4z,x^2y^3z \rangle$. 
\end{center}
\end{example}
We denote $R=k[x,y,z]$, as before.
Then compute the following list of relevant syzygies:
\begin{align}
m_{12}=&\frac{x^2z^3}{xz^3}=x,\nonumber\\
m_{13}=&\frac{x^2y^4z^2}{xy^4z}=xz,\nonumber\\
m_{23}=&\frac{xy^4z^3}{xy^4z}=z^2,\nonumber\\
m_{14}=&\frac{x^2y^3z^2}{x^2y^3z}=z,\nonumber\\
m_{24}=&\frac{x^2y^3z^3}{x^2y^3z}=z^2,\nonumber\\
m_{34}=&\frac{x^2y^4z}{x^2y^3z}=y.\nonumber
\end{align}
Based on the syzygy method we have 
\begin{eqnarray}\label{syzygy equality3}
{\rm HF}\{M\}&=&{\rm HF}\{R\}-{\rm HF}\{R({\rm -deg}(x^2z^2))\}-{\rm HF}\{R/{\langle x  \rangle}  (-{\rm deg}(xz^3))\}\nonumber\\
& &-{\rm HF}\{R/{\langle xz,z^2  \rangle} (-{\rm deg}(xy^4z))\}-{\rm HF}\{R/{\langle z,z^2,y  \rangle} (-{\rm deg}(x^2y^3z))\}\nonumber\\&=&{\rm HF}\{R\}-{\rm HF}\{R(-4)\}-{\rm HF}\{R/{\langle x  \rangle}  (-4))\}-{\rm HF}\{R/{\langle xz,z^2 \rangle} (-6)\}\nonumber\\
& &-{\rm HF}\{R/{\langle z,z^2,y  \rangle} (-6)\}.
\end{eqnarray}
From (\ref{syzygy equality3}) we can see that 
\begin{equation}\label {syzygy term1}
{\rm HF}\{R/{\langle x \rangle} (-4)\}\cong {\rm HF}\{k[y,z](-4)\}
\end{equation}
and 
\begin{equation}\label{syzygy term2}
{\rm HF}\{R/{\langle z,z^2,y \rangle} (-6)\}\cong {\rm HF}\{R/{\langle y,z \rangle} (-6)\}\cong {\rm HF}\{k[x](-6)\}. 
\end{equation}
Therefore, equation(\ref{syzygy term1}) is given by the $2^{\rm nd}$ row of the Pacal table shifted down by four and equation (\ref{syzygy term2}) is given by the $1^{\rm st}$ row of the Pascal table shifted down by six.

Moreover, in order to find the Hilbert function of $M$ we need to find the ${\rm HF}\{R/{\langle xz,z^2 \rangle} (-6)\}$.
Applying again the syzygy method to the fourth summand on the right hand side of (\ref{syzygy equality3})   we have that 
\begin{center}
$m_{12}=\frac{xz^2}{z^2}=x$.
\end{center}
Observe that the shifting is equally distributed in all the terms as follows
\begin{eqnarray} \label{syzygy term5}
{\rm HF}\{R/\langle xz,z^2 \rangle(-6)\}& = &{\rm HF}\{R(-6)\}-{\rm HF}\{R({\rm -deg}(xz))(-6)\}-{\rm HF}\{R/{\langle x  \rangle}  (-{\rm deg}(z^2))(-6)\}\nonumber\\&=&{\rm HF}\{R(-6)\}-{\rm HF}\{R(-2)(-6)\}-{\rm HF}\{R/{\langle x  \rangle}  (-2)(-6)\}\nonumber\\&=&{\rm HF}\{R(-6)\}-{\rm HF}\{R(-8)\}-{\rm HF}\{R/{\langle x  \rangle}  (-8)\}\nonumber\\&=&{\rm HF}\{R(-6)\}-{\rm HF}\{R(-8)\}-{\rm HF}\{k[y,z] (-8)\}\nonumber 
\end{eqnarray}
This way we obtain the following row-generating table
\begin{center}
\scalebox{0.9}{
\begin{tabular}{ c|c|c|c}
${\rm HF}\{R(-6)\}$ & ${\rm -HF}\{R(-8)\}$ & ${\rm -HF}\{k[y,z] (-8)\}$ & ${\rm HF}\{R/\langle xz,z^2 \rangle(-6)\}$\\[-3pt]
\hline
 0 & 0  & 0 &  $0$ \\[-3pt]
\hline
 0 &0  &0 &  $0$\\[-3pt]
\hline
 0 &0  &  0 & $ 0$ \\[-3pt]
\hline
 0 &0  &0 & $ 0$ \\[-3pt]
\hline
 0 & 0  &0 & $0$ \\[-3pt]
\hline
 0 & 0  &  0 & $0$ \\[-3pt]
\hline
 1 & 0  &0 & $1$ \\[-3pt]
\hline
 3 &0  &0 & $3$ \\[-3pt]
\hline
 6 &-1  &-1 & $4$ \\[-3pt]
\hline
 10 & -3  & -2 & $5$ \\[-3pt]
\hline
 15 & -6  &  -3 & $6$ \\[-3pt]
\hline
.. & ..   &.. &.. \\[-3pt]
\hline
.. & ..   &.. &.. 
\end{tabular}}
\end{center}
\vspace{4mm}
Substituting now (\ref{syzygy term1}),(\ref{syzygy term2})  as well as the last column of the above table into (\ref{syzygy equality3}),  we compute the Hilbert function of $M$ \\[7pt]

\begin{center}
\scalebox{0.82}{
\begin{tabular}{ c | c | c | c | c | c }
${\rm HF}\{R\}$ & ${\rm -HF}\{R(-4)\}$ & ${\rm -HF}\{k[y,z](-4)\}$ & ${\rm -HF}\{R/\langle xz,z^2\rangle(-6)\}$ & ${\rm -HF}\{k[x](-6)\}$ & ${\rm HF}\{M\}$\\
\hline
 1 & 0 & 0 & 0 & 0 &$\bf 1$ \\
\hline
 3 & 0 & 0 & 0 & 0 &$\bf 3$ \\
\hline
  6 & 0 & 0 & 0 & 0 &$\bf 6$ \\
\hline
  10 & 0 & 0 & 0 & 0 &$\bf 10$ \\
\hline
 15 & 1 & 1 & 0 & 0 &$\bf 13$ \\
\hline
 21& -3 & -2 & 0 & 0 &$\bf 16$ \\
\hline
  28 & -6 & -3 & -1 & -1 &$\bf 17$ \\
\hline
  36 &-10 & -4 & -3 & -1 &$\bf 18$ \\
\hline
 45 & -15 & -5 & -4 & -1 &$\bf 20$ \\
\hline
  55 & -21 & -6 & -5 & -1 &$\bf 22$ \\
\hline
 66 & -28 & -7 & -6 & -1 &$\bf 24$ \\
\hline
 .. & .. & .. & ..&..&.. \\
\hline
  .. & .. & .. & ..&..&.. 
\end{tabular}}\\[15pt]
\end{center}


\section{Syzygy method via homological algebra}
The short exact sequence that involves $\phi_{x_a}:=\text{multiplication by }x_a$ (see  \cite{villarreal2015monomial} page 98) works well with the assemblage row-by-row of a Hilbert function table. That is because the key homomorphism in the short exact sequence is multiplication by a variable followed by natural projection. Consequently, the last non-zero object of the short exact sequence is the cokernel of $\phi_{x_a}$. This cokernel as we saw in section 2, turns out to be the quotient ring corresponding to the row in the Hilbert function table immediately preceding the introduction of the variable $x_a$. In other words, of the two Hilbert function sequences that the short exact sequence needs to generate the the Hilbert function of $k[x_1,x_2,...,x_a]/I_a$, one of them (the right-most) is the Hilbert function of  $k[x_1,x_2,...,x_{a-1}]/I_{a-1}$. Therefore, any remaining difficulty would be confined to finding the Hilbert function for the kernel  of $\phi_{x_a}$.

In this section, we make use of the same set up as in section 2. Let $S=\{p_1,p_2,p_3,....,p_r\}$, where $p_1,p_2,p_3,...,p_r$ are monomials in the variables $x_1,x_2,...x_d$. Extend this set of variables to an infinite set of variables $x_1,x_2,...,x_d,x_{d+1},...$. For an integer value $a$, let $S_a=\{p_i \in S:p_i \in k[x_1,x_2,...,x_a]\}$. Re-index, if necessary, the set $S$ such that
\begin{enumerate}
\item{$S_a' \subset S_a$ if $a' \geq a$}
\item{For $p_i,\,p_j \in S_a$, $j > i$ only if the highest power of $x_a$ dividing $p_i$ also divides $p_j$.}
\end{enumerate}   
The reader should observe that the first requirement of this re-indexing of the generators of $I$ has the purpose of introducing the generators for the ideals $I_a$ in consecutive order as the variables $x_a$ are introduced one-by-one.  The second criteria for the re-index ensures that, as the set $S_{a-1}$ is enlarged to $S_a$, the new monomials are ordered in (non-strict) increasing order of the power of $x_a$.  This second criteria is done to ensure that the variable $x_a$ does not appear in the syzygies we might need to compute as we generate the $a^{\rm th}$-row of the Hilbert table.  Also, observe that if $S_a=\emptyset$ then set $I_a=0$; otherwise set $I_a= \langle p_i \, | \, p_i \in S_a\rangle$.  Let $M_a=k[x_1,x_2,...,x_a]/I_a$. Construct an infinite array whose ${\rm a}^{th}$ row is the sequence of Hilbert function values of $M_a$.  

Consider the following \emph{short exact sequence} where $\phi_{x_a}$ is multiplication by $x_a$, the module $M_a=k[x_1, x_2, \ldots, x_a]/I_a$, and $(0:x_a)_{M_a}=\ker \phi_{x_a}$ ,
$$0 \rightarrow (0:x_a)_{M_a}(-1)\rightarrow M_a(-1)\overset{\phi_{x_a}}\rightarrow M_a\rightarrow M_a/x_aM_a\rightarrow 0.$$
Set $S_0=\emptyset$ and for $a \geq 1$, if $S_{a-1} \subsetneq S_a$ set  $$U_{x_a}=\{q_i=\frac{p_i}{x_a} \, : \,p_i \in S_a \setminus S_{a-1}\}.$$
If $S_{a-1}=S_a$ then set $U_{x_a}=\emptyset$.
\begin{lemma}
$(0:x_a)_{M_a}=\langle q_i:q_i \in U_{x_a} \rangle _{M_a}$.
\end{lemma}
\begin{proof}
If $U_{x_a}=\emptyset \Leftrightarrow x_a \nmid p_i$ for all $p_i \in S_a \Leftrightarrow \forall g \in M_a, g \neq 0$ then $x_ag \neq 0$.\\
If $ U_{x_a} \neq \emptyset$ the following equivalence holds:
\begin{align}
x_ag=0 \text { in } M_a &\Leftrightarrow p_i \mid x_ag \text{ for some }p_i \in S_a \setminus S_{a-1}\nonumber\\
 &\Leftrightarrow q_i \mid g \text{ for some } q_i \in U_{x_a}\nonumber\\
 &\Leftrightarrow g \in \langle q_i \,:\,q_i \in U_{x_a}\rangle\nonumber
\end{align}
\end{proof}
\begin{remark}
Observe that if $U_{x_a}= \emptyset$ then from the above lemma follows that $(0:x_a)_{M_a}=0$.
\end{remark}
Using the same notation for syzygies as in the previous section, namely $m_{ij}=\frac{{\rm lcm} (p_i, p_j)}{p_j}$ we now state the following lemma.
\begin{lemma}
A non-zero monomial $g \in (0:x_a)_{M_a}$ can be written as follows for one and only one $q_i \in U_{x_a}$,
\begin{enumerate}
\item{$g=\alpha_1q_1$ if $q_1 \in U_{x_a}$}
\item{$g=\alpha_jq_j$ if  $q_j \in U_{x_a}$ and $m_{ij}\nmid \alpha_j$ for all $1 \geq i < j$}
\end{enumerate}
and conversely any $g$ satisfying one of the equations above, belongs to $(0:x_a)_{M_a}$.
\end{lemma}
\begin{proof}
By the previous lemma all we are left to show is uniqueness.\\
Suppose $g \in (0:x_a)_{M_a}$, let $i$ be the smallest index such that $q_i \mid g$.  Then for any $1 \geq i' < i$, $g$ cannot be written as $g=\alpha_{i'}q_{i'}$.\\
If $i < j$ and $q_j \mid g$ then 
\begin{align}
\alpha_{j}=\frac{g}{q_j}\text{ but }q_i \mid g \text{ and }q_j \mid g &\Rightarrow {\rm lcm}(q_i, q_j) \mid g\nonumber\\
&\Leftrightarrow \frac{{\rm lcm} (q_i, q_j)}{q_j} \mid \frac{g}{q_j}\Leftrightarrow m_{ij} \mid \alpha_j.\nonumber
\end{align}  
Therefore, $\alpha_j$ does not satisfy condition 2.
\end{proof}

\begin{theorem}\label{thm}
With the notation of the two lemmas above, the Hilbert function of the annihilator of the homomorphism $\phi_{x_a}$ satisfies the following formula,
\begin{align}
{\rm HF}\{(0:x_a)_{M_a}\}=&\delta(a){\rm HF}\{k[x_1, x_2, \ldots, x_{a-1}](- \deg q_1)\}\nonumber\\
&+\sum_{1<j \in \text{ Index Set } U_{x_a}}\hspace{-5mm} {\rm HF}\{k[x_1, x_2, \ldots, x_{a-1}]/\langle m_{1j}, m_{2j}, \ldots, m_{(j-1)j}\rangle(- \deg q_j)\}\nonumber,
\end{align}
where
$\delta(a)=0$ for $q_1 \notin U_{x_a}$, and $\delta(a)=1$ for $q_1 \in U_{x_a}$.
\end{theorem}
\begin{proof}
If $q_1 \in U_{x_a}$ and $q_1 \mid g$ then $\alpha_1 \in k[x_1, x_1, \ldots, x_{a-1}]/I_{a-1}$ and $\deg(a_1)=b - \deg(q_1)$.  But since $U_{x_{a-1}}=\emptyset$, then $I_{a-1}=0$ which gives us the summand with $\delta(a)=1$.\\
If $q_1 \notin U_{x_a}$ then $\delta(a)=0$ and the first summand is irrelevant.\\
Moreover, for all $g \in (0:x_a)_{M_a}$ expressible as $g=\alpha_jq_j$ with $m_{ij} \nmid \alpha_j$, $1 \geq i < j$, then 
\begin{align}
\alpha_j \in &\left(k[x_1, x_2, \ldots, x_{a-1}]/I_{a-1} \right)/\langle m_{1j}, m_{2j}, \ldots, m_{(j-1)j}\rangle\nonumber\\ 
&\cong k[x_1, x_2, \ldots, x_{a-1}]/\langle m_{1j}, m_{2j}, \ldots, m_{(j-1)j}\rangle.\nonumber
\end{align}
The last isomorphism being due to the second and third isomorphism theorems.
\end{proof}
\begin{remark}
Observe that if $U_{x_a}= \emptyset$ then the sum in the theorem is zero, i.e. ${\rm HF}\left((0:x_a)_{M_a},b\right)=0$ for all $b \geq 0$.
\end{remark}
With the Hilbert function for the annihilator $(0:x_a)_{M_a}$ and the Hilbert function for $M_a/x_aM_a \cong M_{a_1}$ in hand, it is straightforward to implement the procedure outlined in section 2 to generate the Hilbert function of $M_a$.  For that reason, we only show in the next example how to write the Hilbert function of the annihilator in terms of the Hilbert function of simpler quotient rings.

\begin{example}
Use the theorem \ref{thm} to write a sum equivalent to the non-trivial annihilator ideals $(0:x_a)_{M_a}$ where $I=\langle y^6, x^3y^5, x^2y^2z^2, x^3z, x^2yz^3 \rangle$ and the variables are ordered $y, x, z, w_1, w_2, \ldots$.
\end{example}
Before embarking in the computations, we check if the set of generators of $I$ needs re-indexing given the order we have chosen to introduce the variables (this order is quirky in that the variable $y$ is introduced before the variable $x$  and was chosen to illustrate that we are free to select the order in which the variables are introduced).  The criteria that  $S_a' \subset S_a$ if $a' \geq a$ is satisfied by the order in which the monomials generating $I$ are listed.  However, the second criteria; namely, that for $p_i,\,p_j \in S_a$, $j > i$ only if the highest power of $x_a$ dividing $p_i$ also divides $p_j$, requires that the order of the monomials $x^2y^2z^2$ and $x^3z$ be swapped.  Observe that adjusting the indexing to satisfy the second criteria does not interfere with the first criteria.  In other words, after swapping the third and fourth monomials we get $I=\langle y^6, x^3y^5, x^3z, x^2y^2z^2, x^2yz^3 \rangle$ which satisfies both re-indexing criteria.
The first annihilator is $(0:y)_{M_y}$, where $M_y = k[y]/\langle y^6 \rangle$. In this case, $U_y=\{y^5\}$ and ${\rm HF}\{(0:y)_{M_y}\}={\rm HF}\{k(-5)\}.$
The second annihilator is $(0:x)_{M_x}$, where $M_x=k[y, x]/\langle y^6, x^3y^5\rangle$ and $U_x=\{x^2y^5\}$. There is only the syzygy $m_{12}=y$ to consider.  Therefore,
$${\rm HF}\{(0:x)_{M_x}\}={\rm HF}\{k[y]/\langle y\rangle(- 7)\}={\rm HF}\{k(- 7)\}.$$
The third and last non-trivial annihilator is $(0:z)_{M_z}$, where $$M_z=k[y, x, z]/\langle y^6, x^3y^5, x^3z, x^2y^2z^2, x^2yz^3 \rangle.$$  In this case, $U_z=\{x^3, x^2y^2z, x^2yz^2\}$.  The syzygies to consider are 
\begin{align}
&m_{13}=y^6 &&m_{23}=y^5 &&\phantom{m_{13}=y^4} &&\phantom{m_{23}=x^3y^3}\nonumber\\
&m_{14}=y^4 &&m_{24}=xy^3 &&m_{34}=x  &&\phantom{m_{23}=x^3y^3}\nonumber\\ 
&m_{15}=y^5 &&m_{25}=xy^4 &&m_{35}=x  &&m_{45}=y.\nonumber
\end{align}
Therefore,
\begin{align}
{\rm HF}\{(0:z)_{M_z}\}=&{\rm HF}\{k[y,x]/\langle y^6, y^5\rangle(- 3)\}\nonumber\\
&+{\rm HF}\{k[y,x]/\langle y^4, xy^3, x\rangle(- 5)\}\nonumber\\
&+{\rm HF}\{k[y,x]/\langle y^5, xy^4, x, y\rangle(- 5)\}\nonumber\\
=&{\rm HF}\{k[y,x]/\langle y^5\rangle(- 3)\}\nonumber\\
&+{\rm HF}\{k[y]/\langle y^4\rangle(- 5)\}+{\rm HF}\{k(- 5)\}.\nonumber
\end{align}

\section{Conclusion}
As the reader can see the Syzygy method via homological algebra is quite close in spirit to the Syzygy method discussed in the previous section.  The only significant difference is in the tools used to prove it.  Therefore, all the information about the Hilbert function was obtain from the syzygies.  Last but not least, we have developed two different approaches to computing the Hilbert function of a quotient ring: the lcm-lattice method and the syzygy method.

\bibliographystyle{unsrt}
\bibliography{references}

\begin{thebibliography}{1}

\bibitem{villarreal2015monomial}
Rafael Villarreal.
\newblock {\em Monomial algebras}.
\newblock Chapman and Hall/CRC, 2015.

\bibitem{dummit2004abstract}
David~Steven Dummit and Richard~M Foote.
\newblock {\em Abstract algebra}, volume~3.
\newblock Wiley Hoboken, 2004.

\bibitem{eisenbud2013commutative}
David Eisenbud.
\newblock {\em Commutative Algebra: with a view toward algebraic geometry},
  volume 150.
\newblock Springer Science \& Business Media, 2013.

\bibitem{W:B&W}
Margherita Barile and Eric~W. Weisstein.
\newblock {Hilbert Function website}.
\newblock \url{ http://mathworld.wolfram.com/HilbertFunction.html}.

\bibitem{atiyah2018introduction}
Michael Atiyah.
\newblock {\em Introduction to commutative algebra}.
\newblock CRC Press, 2018.

\bibitem{miller2004combinatorial}
Ezra Miller and Bernd Sturmfels.
\newblock {\em Combinatorial commutative algebra}, volume 227.
\newblock Springer Science \& Business Media, 2004.

\bibitem{van2001course}
Jacobus~Hendricus Van~Lint and Richard~Michael Wilson.
\newblock {\em A course in combinatorics}.
\newblock Cambridge university press, 2001.

\end{thebibliography}

\end{document}